\documentclass[11pt, reqno]{amsart}
\usepackage{etex}
\usepackage{textcmds} 
\usepackage{amsmath, amssymb, amsfonts, amstext, verbatim, amsthm, mathrsfs}
\usepackage{microtype}
\usepackage[all]{xy}
\usepackage[modulo]{lineno}
\usepackage[usenames]{color}
\usepackage[usenames,dvipsnames]{xcolor}
\usepackage{aliascnt}
\usepackage{enumitem}
\usepackage{xspace}
\usepackage{amsfonts}
\usepackage[centertags]{amsmath}
\usepackage{rotating}
\usepackage[margin=3.5cm]{geometry}
\usepackage{dsfont}
\usepackage{bm}
\usepackage{color}
\usepackage{subfigure}
\usepackage{amsmath}
\usepackage{array}
\usepackage[all]{xy}
\usepackage{euscript}
\usepackage[T1]{fontenc}
\usepackage{mathbbol}
\usepackage{nccmath}
\usepackage{marginnote}
\usepackage{stmaryrd}
\usepackage{youngtab}
\usepackage{tikz,pgfplots}
\usepackage{blkarray}
\usepackage{ulem}
\usepackage[abs]{overpic}

\usepackage{graphics,graphicx}  %% comment for "draft" mode w/o pictures
\let\oldtocsection=\tocsection
\let\oldtocsubsection=\tocsubsection

\renewcommand{\tocsection}[2]{\hspace{0em}\oldtocsection{#1}{#2}}
\renewcommand{\tocsubsection}[2]{\hspace{1em}\oldtocsubsection{#1}{#2}}

\usepackage[backref,colorlinks=true,linkcolor=black,citecolor=blue,urlcolor=blue,citebordercolor={0 0 1},urlbordercolor={0 0 1},linkbordercolor={0 0 1}]{hyperref} %needs to be 

\tikzset{node distance=3cm, auto}

\makeatletter
\def\@secnumfont{\bfseries}
\def\section{\@startsection{section}{1}%
  \z@{.7\linespacing\@plus\linespacing}{.5\linespacing}%
  {\normalfont\Large\bfseries}}
\def\subsection{\@startsection{subsection}{2}%
  \z@{.5\linespacing\@plus.7\linespacing}{-.5em}%
  {\normalfont\large\bfseries}}
  \def\subsubsection{\@startsection{subsubsection}{3}%
  \z@{.5\linespacing\@plus.7\linespacing}{-.5em}%
  {\normalfont\bfseries}}
\makeatother

\newtheorem{thm}{Theorem}[section]
\newtheorem{lemma}[thm]{Lemma}
\newtheorem{prop}[thm]{Proposition}
\newtheorem{cor}[thm]{Corollary}

\newtheorem{remark}[thm]{Remark}
\newtheorem{rmk}[thm]{Remark}

\newtheorem{conjecture}[thm]{Conjecture}
\newtheorem{definition}[thm]{Definition}

\newtheorem{example}[thm]{Example}
\theoremstyle{remark}
\numberwithin{equation}{subsection} 

\numberwithin{figure}{section}
\numberwithin{table}{subsection}

\newcommand{\oCP}{{\overline{{\C}P}}\!\,}

\newcommand{\N}{{\mathbb{N}}}
\newcommand{\Ee}{\mathcal{E}}

\newcommand{\Tt}{\mathcal{T}}

\newenvironment{itemlist}
   { \begin{list} {$\bullet$}
         { \setlength{\topsep}{.5ex}  \setlength{\itemsep}{.5ex} \setlength{\leftmargin}{2.5ex} } }
   { \end{list} }
   
    \newcommand{\vn}{{\vec{n}}}
 \newcommand{\oxr}{{\ov{x}\rho}}
 \newcommand{\oyl}{{\ov{y}\la}}
 \newcommand{\ovx}{{\ov{x}}}
 \newcommand{\ovy}{{\ov{y}}}
  \newcommand{\ovr}{\overrightarrow}

\newcommand{\bbm}{{\bf{m}}}
\newcommand{\bw}{{\bf{w}}}
\newcommand{\bx}{{\bf{x}}}

\newcommand{\bE}{{\bf{E}}}
\newcommand{\bB}{{\bf{B}}}

\newcommand{\NI}{{\noindent}}

\newcommand{\Cc}{{\mathcal C}}

\newcommand{\ov}{\overline}

\newcommand{\la}{{\lambda}}

\newcommand{\MS}{{\medskip}}

\newcommand{\er}{{\Diamond}}
\newcommand{\Z}{\mathbb{Z}}

\newcommand{\R}{\mathbb{R}}

\newcommand{\C}{\mathbb{C}}
\newcommand{\CP}{\mathbb{CP}}

\newcommand{\eps}{\varepsilon}

\newcommand{\E}{\mathcal{E}}

\newcommand{\acc}{\mathrm{acc}}
\newcommand{\sembeds}{\stackrel{s}{\hookrightarrow}}
\newcommand{\se}{\stackrel{s}{\hookrightarrow}}

\makeatletter
\newcommand{\dashover}[2][\mathop]{#1{\mathpalette\df@over{{\dashfill}{#2}}}}
\newcommand{\fillover}[2][\mathop]{#1{\mathpalette\df@over{{\solidfill}{#2}}}}
\newcommand{\df@over}[2]{\df@@over#1#2}
\newcommand\df@@over[3]{%
  \vbox{
    \offinterlineskip
    \ialign{##\cr
      #2{#1}\cr
      \noalign{\kern1pt}
      $\m@th#1#3$\cr
    }
  }%
}
\newcommand{\dashfill}[1]{%
  \kern-.5pt
  \xleaders\hbox{\kern.5pt\vrule height.4pt width \dash@width{#1}\kern.5pt}\hfill
  \kern-.5pt
}
\newcommand{\dash@width}[1]{%
  \ifx#1\displaystyle
    2pt
  \else
    \ifx#1\textstyle
      1.5pt
    \else
      \ifx#1\scriptstyle
        1.25pt
      \else
        \ifx#1\scriptscriptstyle
          1pt
        \fi
      \fi
    \fi
  \fi
}
\newcommand{\solidfill}[1]{\leaders\hrule\hfill}
\makeatother

\title{Unobstructed Embeddings in Hirzebruch surfaces}
\author{Nicki Magill}
\address{Mathematics Department, Cornell University}
\email{nm627@cornell.edu}
\thanks{NSF Graduate Research Grant DGE-1650441}
\keywords{symplectic embeddings in four dimensions, ellipsoidal capacity function, infinite staircases, almost toric fibrations}
\subjclass{53D05}
\date{\today}

\begin{document}

\begin{abstract} This paper continues the study of the ellipsoid embedding function of symplectic Hirzebruch surfaces parametrized by $b \in (0,1)$, the size of the symplectic blowup. Cristofaro-Gardiner, et al. (arxiv: 2004.13062) found that if the embedding function for a Hirzebruch surface has an infinite staircase, then the function is equal to the volume curve at the accumulation point of the staircase. Here, we use almost toric fibrations to construct full-fillings at the accumulation points for an infinite family of recursively defined irrational $b$-values implying these $b$ are potential staircase values. The $b$-values are defined via a family of obstructive classes defined in Magill-McDuff-Weiler (arxiv:2203.06453). There is a correspondence between the recursive, interwoven structure of the obstructive classes and the sequence of possible mutations in the almost toric fibrations. This result is used in Magill-McDuff-Weiler (arxiv:2203.06453) to show that these classes are exceptional and that these $b$-values do have infinite staircases. 
\end{abstract}
\maketitle
\tableofcontents

\section{Introduction} 

Given a closed 4-dimensional symplectic manifold $(X,\omega),$ we define its {\bf ellipsoid embedding function} to be 
\[ c_X(z):=\inf \{ \lambda \ | \ E(1,z) \se \lambda X\},\]
where $z \geq 1$ is a real variable, $\lambda X:=(X,\lambda \omega)$ is the symplectic scaling, and the ellipsoid $E(1,z)\subset \C^2$ is the set
\[ E(1,z)=\{ (\zeta_1,\zeta_2) \in \C^2 \ | \ \pi\left(|\zeta_1|^2+\frac{|\zeta_2|^2}{z}\right)<1\}. \] 

Here, we consider the case where the target, $X=H_b:=\CP^2_1 \# \oCP^2_b,$ is a one-fold blowup of $\CP^2$ where the line has area $1$ and the exceptional divisor has area given by the parameter $0 \leq b<1.$ Let $c_b(z):=c_{H_b}(z)$ denote the embedding function for $H_b$. This function has been previously studied in \cite{ICERM,MM,ball,AADT}. There are many papers about the ellipsoid embedding function for other targets including \cite{CG, FM,Usher}.

As symplectic embeddings preserve volume, the lower bound for the embedding function
\[ c_{b}(z) \geq \sqrt{\frac{z}{1-b^2}}=:V_b(z) \]
is immediate, where $1-b^2$ is the appropriately normalized volume of $H_b.$

As shown in \cite[Prop.2.1]{AADT}, $c_{b}(z)$ is continuous, non-decreasing, and piecewise linear when not equal to the volume curve $V_b(z)$. We say $H_b$ has an {\bf infinite staircase} if $c_{b}(z)$ has infinitely many nonsmooth points. Furthermore, by \cite{BHO} and \cite{AADT}, $H_b$ satisfies packing stability, which states that for large enough $z$, $c_{b}(z)$ is equal to the volume curve. Hence, if there is an infinite staircase, the nonsmooth points must accumulate at some finite value denoted $\acc(b).$ Prior work including \cite{AADT}, \cite{ICERM}, and \cite{MM} established that there exist infinitely many $b$-values with infinite staircases. A complete classification of $b \in [0,1)$ that have infinite staircases remains unknown. The first result that there is a target with an infinite staircase is the case $b=0$, i.e. the target is $\CP^2$, found in \cite{ball}. In \cite{AADT}, the authors conjectured that $b=0$ and $b=1/3$ are the only rational $b$ for which $c_b(z)$ has an infinite staircase. 

In \cite[Theorem 1.11]{AADT}, the authors identified a useful criterion for finding potential $b$-values with infinite staircases. Namely, if there is an infinite staircase accumulating at $\acc(b)$, $\acc(b)$ must be the larger solution to: 
\[z^2-\left(\frac{(3-b)^2}{1-b^2}-2\right)z+1=0,\] 
where $3-b$ is the affine perimeter of the corresponding Delzant polygon and $1-b^2$ is the symplectic volume of $H_b$. Further, if there is an infinite staircase, then 
\begin{equation}  \label{eq:accVol} c_{b}(\acc(b))=\sqrt{\frac{\acc(b)}{1-b^2}}=V_b(\acc(b)). \end{equation}

\begin{center}
\begin{figure}[ht]
\begin{tikzpicture}
\begin{axis}[
	axis lines = middle,
	xtick = {5.5,6,6.5,7},
	ytick = {2.5,3,3.5},
	tick label style = {font=\small},
	xlabel = $z$,
	ylabel = $y$,
	xlabel style = {below right},
	ylabel style = {above left},
	xmin=5.4,
	xmax=7.2,
	ymin=2.4,
	ymax=3.6,
	grid=major,
	width=3in,
	height=2.25in]
\addplot [red, thick,
	domain = 0:0.7,
	samples = 120
]({ (((3-x)*(3-x)/(2*(1-x^2))-1)+sqrt( ((3-x)*(3-x)/(2*(1-x^2))-1)* ((3-x)*(3-x)/(2*(1-x^2))-1) -1 ))},{sqrt(( (((3-x)*(3-x)/(2*(1-x^2))-1)+sqrt( ((3-x)*(3-x)/(2*(1-x^2))-1)* ((3-x)*(3-x)/(2*(1-x^2))-1) -1 )))/(1-x^2))});
% \addplot [black, only marks, very thick, mark=*] coordinates{(6,2.5)};
\addplot [blue, only marks, very thick, mark=*] coordinates{(6.85,2.62)};
\addplot [green, only marks, very thick, mark=*] coordinates{(5.83,2.56)};
\end{axis}
\end{tikzpicture}
\caption{ This paramterized curve was first studied in the work of Bertozzi, et. al, and this figure is first found in \cite[Figure 1.1.4]{ICERM}. The curve shows the location of the accumulation point $(z,y)=(\acc(b), V_b(\acc(b)))$ for $0\le b< 1$.
The blue point with $b=0$ is at $(\tau^4, \tau^2)$ and is the accumulation point for the Fibonacci stairs.  The green point  
is the accumulation point for $b=1/3$, and is the minimum of the function $b\mapsto\acc(b)$.  
The family of $b$-values in Theorem~\ref{thm:Main1} all have $z \geq 7.$ 
} \label{fig:accpt} 
\end{figure}
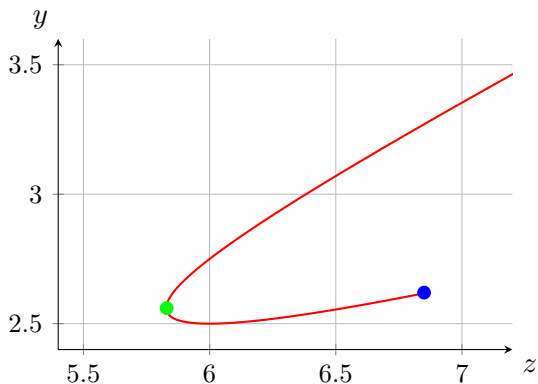
\end{center}

If the condition \eqref{eq:accVol} is satisfied for a particular $b$-value, we say that $b$ is {\bf unobstructed}. On the other hand, if $c_{b}(\acc(b))>V_b(\acc(b))$, we say $b$ is {\bf blocked}. 

In this paper, we find an infinite family of $b$-values for which $b$ is unobstructed by constructing a sequence of almost toric fibrations that limits to a full-filling at the accumulation point. An {\bf almost toric fibration} (ATF) on $H_b$ is a Lagrangian fibration $\pi:H_b \to B$ where $B$ is a 2-dimensional base and there are restrictions on the types of singularities of the fibration. See Section~\ref{sec:ATFdef}  for a more precise definition. A classification given by \cite{LS} of ATFs on general compact four-manifolds without boundary implies that we can represent ATFs on $H_b$ by decorated quadrilaterals, known as almost toric base diagrams in the literature; see Definition~\ref{def:Q} and Figure~\ref{fig:basediagr}. The decorations on the quadrilaterals denote the various singularities of the fibration. For general background on ATFs, we recommend \cite{S} and \cite{evanBook}.

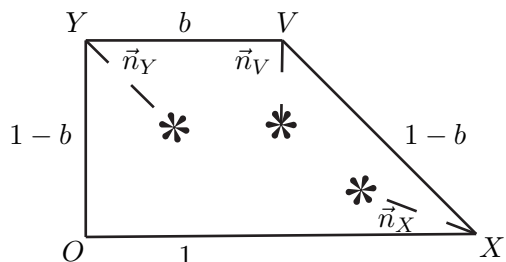
\begin{figure}
    \centering
     \begin{overpic}[scale=1,unit=0.5mm]{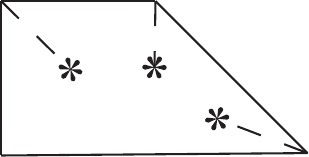}
 	\put (-6,-6) {$O$}
  \put (105,-5) {$X$}
   \put (51,55) {$V$}
    \put (-5,55) {$Y$}
    \put (10,45) {$\vn_Y$}
    \put (40,45) {$\vn_V$}
    \put (78,3) {$\vn_X$}
    \put (-20,25) {$1-b$}
    \put (25,-7) {$1$}
    \put (85,25) {$1-b$}
    \put (25,55) {$b$}
  \end{overpic}
    \caption{This is the decorated quadrilateral $Q_b:=OXVY$. It is given by the moment polygon for $H_b$ with the three nodal rays $\vn_Y = (-1,1),\vn_V = (0,-1),$ and $\vn_X = (-2,1)$ inserted.
    }
    \label{fig:basediagr}
\end{figure}

We begin with the base diagram $Q_b$ for $H_b$ defined in Figure~\ref{fig:basediagr}. This base diagram is not unique, and we can perform sequences of mutations to $Q_b$ to arrive at different base diagrams for the same ATF on $H_b.$ The three mutations we consider, see Definition~\ref{def:mutation}, on a quadrilateral $Q:=OXVY$ are denoted $x,v,y$ where for $w=x,y,z$, $w:Q \to wQ=OX_wV_wY_w$ is a new quadrilateral. These mutations can be composed with each other. See Figure~\ref{fig:yMutIntro} for an example of the mutation $y.$

Following the authors of \cite{CV} and \cite{AADT}, sequences of mutations construct sequences of symplectic embeddings. In particular, 
let $Q=OXVY$ be a decorated quadrilateral corresponding to an ATF on $H_b$ where $O$ is the origin, $X$ is on the $x$-axis, and $Y$ is on the $y$-axis. Then, following work of Symington in \cite{S}, \cite[Proposition 2.2.7]{AADT} states that for any $0<\eps<1,$
 \[(1-\eps)E\left(|OX|,|OY|\right) \sembeds H_b.\] See Figure~\ref{fig:yMutIntro} to see how applying a mutation gives rise to a different embedding. 

  \begin{center}
 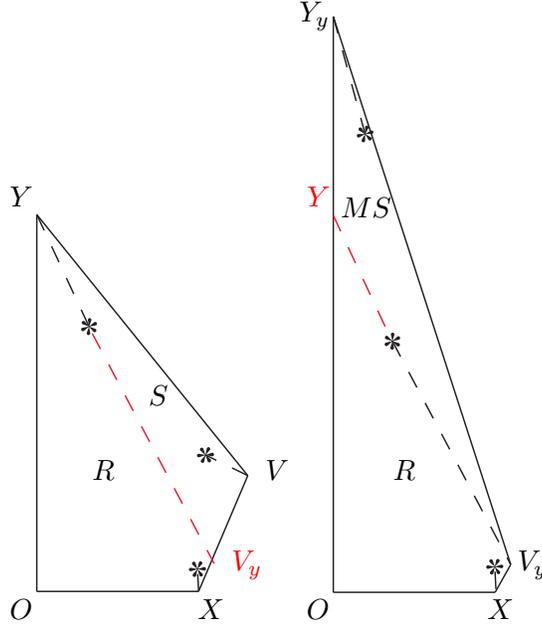
\begin{figure}
 \begin{overpic}[scale=0.55,unit=0.5mm]{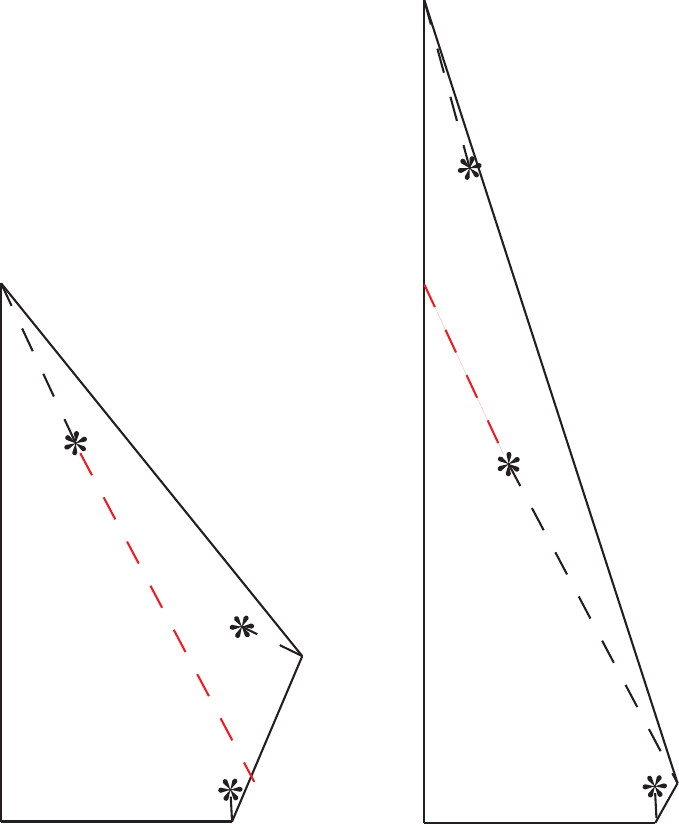}
 	\put (15,30) {$R$}
  \put (95,30) {$R$}
  \put (30,50) {$S$}
   \put (81,100) {$MS$}
   \put (-7,-7) {$O$}
   \put (72,-7) {$O$}
   \put (43,-7) {$X$}
   \put (61,30) {$V$}
   \put (-7,103) {$Y$}
   \put (52,6) {{\color{red} $V_y$}}
   \put (120,-7) {$X$}
   \put (128,6) {$V_y$}
   \put (70,152) {$Y_y$}
   \put (72,103) {{\color{red} $Y$}}
  \end{overpic}
     \caption{On the left is $Q:=OXVY$, an ATF base diagram for $H_b$. On the right is $yQ:=OXV_yY_y.$ $M$ is an integral affine transformation that fixes the nodal ray at $Y$ and $M(YV_yV)=YV_yY_y$. The quadrilateral $Q$ implies $E\left(|OX|,|OY|\right) \sembeds H_b$, and $yQ$ implies $E\left(|OX|,|OY_y|\right) \sembeds H_b$. As $OXY_y$ has a larger area than $OXY$, this embedding fills more volume of $H_b.$}
     \label{fig:yMutIntro}
 \end{figure}
 \end{center} 

We can now state the main theorem, which is proved in Section~\ref{sec:MainProof}. 

\begin{thm} \label{thm:Main1} 
For every word $wyv^{n+2}$ where $w$ is any finite sequence of $x$ and $y$ and $n$ is any non negative integer, there is a $b:=b(w,n)$ such that 
\[ \lim_{k \to \infty} y^kwyv^{n+2}Q_b\] is a triangle. Furthermore, this limiting figure implies  that for this $b$-value \[c_b(\acc(b))=V_b(\acc(b)),\] i.e. $b$ is unobstructed. 
\end{thm}

To prove the theorem, for all such $w$, we compute the data of the base diagram $y^kwyv^{n+2}Q_b$. The formulas for these are given in Definition~\ref{def:QT} and Proposition~\ref{prop:triple} while the majority of the computations are in Section~\ref{sec:computations}. 

In Section~\ref{sec:obs} of this paper, we explain how Theorem~\ref{thm:Main1} is related to a family of exceptional classes, that is symplectic spheres of self intersection $-1$ living in various blowups of $\CP^2$, defined by the author, McDuff, and Weiler in \cite{MMW}. In particular, \cite{MMW} uses Theorem~\ref{thm:Main1} to show this family of exceptional classes are represented by embedded spheres. One way this is used in \cite{MMW} is to show that all the $b$-values in Theorem~\ref{thm:Main1} have infinite staircases. 

For a given word $wyv^{n+2},$ the $b$-value in Theorem~\ref{thm:Main1} can be computed explicitly using Definition~\ref{def:brho}. 
We now give some context on how to visualize the $b$-values in Theorem~\ref{thm:Main1}. 
For a fixed $n$, the $b$-values in Theorem~\ref{thm:Main1} are in the interval $\left(\tfrac{n+1}{n+2},\tfrac{n+2}{n+3}\right)$. The word $w$ then determines where $b$ lies within this interval. The $b$-values are intertwined in a Cantor-like structure depending on if we take $x$ or $y$ mutations. See Figure~\ref{fig:tree} to visualize this and Remark~\ref{rmk:cantor} for more information. 

Unlike in \cite{AADT} and \cite{CV}, the sequence of embeddings constructed via consecutive $y$ mutations do not correspond to optimal embeddings, and thus, do not give us points on the embedding function. Rather, these correspond to a sequence of embeddings that lie strictly above the embedding function on a horizontal line through the accumulation point. See Remark~\ref{rmk:irrb} for more details about this. Regardless, evidence from this work, \cite{M}, and \cite{REU} lead to the following conjecture. 

   \begin{center}
 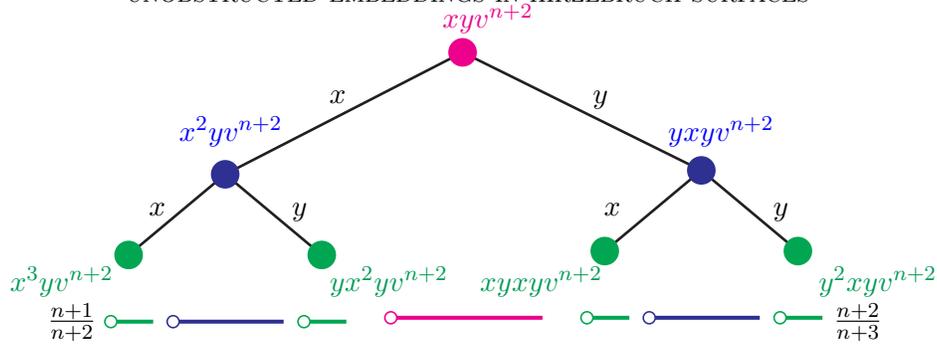
\begin{figure}
 \begin{overpic}[scale=1,unit=0.5mm]{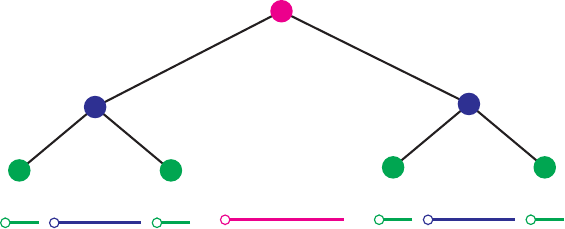}
 	\put (90,80) {\color{Magenta} $xyv^{n+2}$}
    \put (130,60) {$y$}
     \put (60,60) {$x$}
     \put (20,50) {\color{blue} $x^2yv^{n+2}$}
     \put (150,50) {\color{blue} $yxyv^{n+2}$}
      \put (12,30) {$x$}
       \put (50,30) {$y$}
       \put (133,30) {$x$}
       \put (178,30) {$y$}
       \put (190,10) {\color{ForestGreen} $y^2xyv^{n+2}$}
        \put (100,10) {\color{ForestGreen} $xyxyv^{n+2}$}
        \put (-25,10) {\color{ForestGreen} $x^3yv^{n+2}$}
         \put (60,10) {\color{ForestGreen} $yx^2yv^{n+2}$}
         \put (-15,0) {$\frac{n+1}{n+2}$}
         \put (194,0) {$\frac{n+2}{n+3}$}
  \end{overpic}
  \vspace{5mm}
     \caption{For each $n$, we construct a labeled infinite tree as follows. Each vertex $f$ is labeled with some word $w_f$. The first vertex is labeled $xyv^{n+2}.$ At each vertex, there is an edge to the left corresponding to an $x$-mutation and an edge to the right corresponding to a $y$-mutation.  For a vertex $f$, the $b$-value for which $y^kxw_fQ_b$ limits to a triangle corresponds to the left endpoint of the interval directly below that vertex on the number line from $(\frac{n+1}{n+2},\frac{n+2}{n+3})$ drawn below the graph. The intervals are not to scale, but their ordering is correct.}
     \label{fig:tree}
 \end{figure}
 \end{center} 

\begin{conjecture} 
For all $0 \leq b <1$, the function $c_b(z)$ on the interval $[1,\acc(b))$ can be computed from ATF mutations to the base diagram $Q_b.$ 
\end{conjecture} 

The combined work of the authors in \cite{CV} and \cite{AADT} imply the conjecture for $b=0,1/3.$ The sequences of mutations considered in this paper suggest the appropriate word $w$ for other $b$-values. Further, the close relation between the numerics of obstructive classes and the ATF base diagrams seen in Section~\ref{sec:obs} gives insight on how to argue the embeddings from the base diagrams are optimal.

\begin{rmk} \rm {\bf (Future Work)} \label{rmk:future}
(i) Theorem~\ref{thm:Main1} is used in \cite{MMW} to show that for these unobstructed $b$, $c_{b}(z)$ has an infinite staircase. In particular, Theorem~\ref{thm:Main1} is cited in \cite{MMW}[Cor.3.1.5] to show certain homology classes in blowups of $\CP^2$ are represented by embedded spheres of self intersection $-1$ rather than immersed spheres. See Section~\ref{sec:obs} for more information about these homology classes. \MS

\NI (ii) The sequences of ATFs in Theorem~\ref{thm:Main1} was first used by the author in \cite{M} to prove the existence of infinite staircases in a two-fold blowup of $\CP^2$ with irrational blowup sizes. Following this work and \cite{M}, Farley, et. al, in \cite{REU} found that the same sequences of mutations can be used to prove the existence of infinite staircases for polydisks. One crucial aspect of the sequences in these papers is Lemma~\ref{lem:LimitMot} which states that the reciprocal of the volume at the accumulation point is a linear function of $b$ with rational coefficients. A similar statement holds for three and four fold blowups of $\CP^2,$ and thus, we expect similar results to generalize to those domains. Note, as mentioned in Remark~\ref{rmk:irrb}, this precise method only works for irrational blowup sizes. \MS

\NI (iii) In \cite{MM} and \cite{MMW}, the authors found symmetries that take the family of $b$-values in Theorem~\ref{thm:Main1} to other families of $b$-values with similar properties. In this paper, we restrict to looking at $v$-mutations where the nodal ray at $V$, denoted $\vn_V$, intersects the $OX$ side of $Q.$ Evidence suggests that considering $b$-values for when $\vn_V$ sometimes intersects $OX$ and sometimes intersects $OY$ would allow one to compute relevant embeddings for the $b$-values in Theorem~\ref{thm:Main1} acted on by the symmetries.  \MS

\NI (iv) Previous work including \cite{BS} and \cite{ES} have used the ATF mutations to study Lagrangian pinwheels, a class of $2$-dimensional cell complexes, in $\CP^2$. This paper contains lots of numerics about the pinwheels in $H_b.$ It would be interesting to see if similar results found in \cite{BS} and \cite{ES} hold for $H_b.$ 

\end{rmk}

  \NI {\bf Acknowledgements} This paper was written at the same time as \cite{MMW} with Dusa McDuff and Morgan Weiler. I would like to thank both of them for many useful conversations, suggestions, edits, and support. Additionally, I thank Tara Holm as my research advisor for introducing me to almost toric fibrations, encouraging me to use them as a proof method, and helpful edits. I would also like to thank Ana Rita Pires for useful conversations. Finally, the anonymous referee provided and inspired significant suggestions to the exposition and organization of the paper. I thank them for these detailed suggestions and other helpful comments.

  \section{Background and Definitions} \label{sec:ATFdef}
\begin{definition}[\cite{LS},\cite{S}]
    An {\bf almost toric fibration} of a symplectic 4-manifold is a Lagrangian fibration $\pi:(M,\omega)\to B$ such that any critical point of $\pi$ has a Darboux neighborhood $(x_1,y_1,x_2,y_2)$ with symplectic form $\omega_{st}=dx_1 \wedge dy_1+ dx_2 \wedge dy_2$, in which $\pi$ has one of the following local normal forms:
    \begin{align*}
\pi(x,y)&=(x_1,x_2^2+y_2^2), \quad \quad \text{elliptic, corank one} \\
\pi(x,y)&=(x_1^2+y_1^2,x_2^2+y_2^2), \quad \quad \text{elliptic, corank one}  \\
\pi(x,y)&=(x_1y_1+x_2y_2,x_1y_2-x_2y_1) \quad \quad \text{focus-focus} .
\end{align*}
\end{definition}

Given an almost toric fibration $H_b \to B$, we can construct an almost toric base diagram. To see how the base diagram is constructed from the fibration see \cite[Sec.8.1]{evanBook}. The base diagrams for $H_b$ that we consider look something like $Q_b$ defined in Figure~\ref{fig:basediagr}. As explained in \cite[Sec 8.1]{evanBook} and \cite[Sec. 6.2]{S}, $Q_b$ is an ATF base diagram on $H_b$ as it is the moment polygon of $H_b$ with three nodal trades. In the base diagrams we consider, there are three {\bf nodal rays} emanating from three of the vertices with a marked point at the end of the ray. The singularities are marked as follows: 
\begin{itemize}
    \item the marked points are the image of a focus-focus singularity whose preimage is a pinched torus;
    \item all other points in the interior of the polygon are regular values with preimage a torus;
    \item all points on the boundary of the polygon are the image of elliptic singularities with preimage either a point if it is a vertex with no nodal ray and otherwise, a circle.
\end{itemize}

One operation we can perform to a base diagram is a {\bf nodal slide} where we move the marked point along the ray. By \cite[Prop. 6.2]{S}, these base diagrams correspond to two different ATFs on $H_b,$ i.e. the symplectic structure does not change.

The base diagrams we consider will be determined by the data in the following definition. 

\begin{definition}\label{def:Q}
A {\bf decorated quadrilateral} $Q:=OXVY$ is a quadrilateral in $\R^2$ with various data and decorations. The data of a decorated quadrilateral includes the affine lengths\footnote{All vectors we are considering are multiples of some lattice vector. If $\vec{v}$ is such a vector from $A$ to $B$, then the affine length of $\vec{v}$ is the number $\la$ such that $\vec{v}=\la\vec{w}$ where $\vec{w}$ is the primitive lattice vector in the direction of $\vec{v}.$} 
of the sides (written as $|XV|$), the nodal rays $\vn_X$, $\vn_V$, $\vn_Y$ emanating from vertices $X,V,$ and $Y$ as well as the direction vectors $\ovr{XV}$ and $\ovr{VY}.$ Note the directions vectors $\ovr{OX}=(1,0)^T$ and $\ovr{OY}=(0,1)^T$ are fixed for any $Q.$
The nodal rays and direction vectors will always have affine length $1$.   Define $Q_b$ as the base diagram for the Delzant polygon of $H_b$ with three nodal rays inserted at vertices
 $X,V,Y$ where
 \[ \vn_X=\begin{pmatrix} -2 \\ 1 \end{pmatrix}, \quad \vn_V=\begin{pmatrix} 0 \\ -1 \end{pmatrix}, \quad \vn_Y=\begin{pmatrix} -1 \\ 1 \end{pmatrix}.
 \] 
 The data labeled on a diagram can be seen in Figure~\ref{fig:basediagr}. Lastly, for the vertices 
 $W=X,V,$ or $Y,$ define the {\bf monodromy matrix}  $M_W$ to be an integral affine matrix such that
    \[ M_W(\vn_W)=\vn_W, \quad 
    \quad M_Y\ovr{VY}=(0,-1)^T, \quad M_V\ovr{XV}=\ovr{VY}, \quad and \quad M_X\ovr{XV}=(1,0)^T.
    \]
\end{definition} 

Given an ATF base diagram of $H_b$, we can perform mutations to get different base diagrams for the same ATF on $H_b.$ Often, we combine a mutation with a nodal slide, which does correspond to a change of the fibration. For more background on  mutations, see \cite{V1}, \cite{V2}, \cite[Sec.2.4]{AADT}, and Figure~\ref{fig:yMutIntro} for an example of a mutation. 

\begin{definition} \label{def:mutation}
    Let $Q:=OXVY$ be some decorated quadrilateral. The mutations we define are as follows: 
    \begin{itemize}
 \item the mutation $v: Q \mapsto vQ=OX_vV_vY$  at $V$ is the case when $\vn_V$ extends to intersect the side $OX$ of $Q$ at $X_v$, the new position  of $X$ and $M_V(X_vXV)=X_vV_vV$. The nodal ray at $\vn_{X_x}=-\vn_V$;
  \item the mutation $x: Q \mapsto xQ=OX_xV_xY$  at $X$ is the case when $\vn_X$ extends to intersect the side $VY$ of $Q$ at $V_x$, the new position of $V$ and $M_X(V_xVX)=V_xX_xX.$ The nodal ray at $\vn_{V_x}=-\vn_X;$
 \item the mutation $y:Q \mapsto yQ=OXV_yY_y$  at $Y$ is the case when $\vn_Y$ extends to intersect the side $XV$ of $Q$ at $V_y$, the new position of  $V$ and $M_Y(V_yVY)=V_yY_yY$. The nodal ray at $\vn_{V_y}=-\vn_Y.$
 \end{itemize}

\end{definition}
    Note, the definition of these mutations insist the nodal rays extend to intersect the side mentioned in the definition. For instance, the extension of $\vn_V$ could instead intersect the side $OY$, but if this is case, we do not denote this mutation $vQ.$ To compute the image of the triangle in the definitions under the monodromy matrix, note that affine lengths are preserved under this transformation.

Motivated by Theorem~\ref{thm:Main1} where we consider mutations of the form $y^kwyv^{n+2}Q_b$, we begin by computing $yv^kQ_b$ under certain constraints for $b.$ See Figure~\ref{fig:vyMut} for a sketch of $yv^{n+2}Q_b$ for $b$ in this range.  

\begin{lemma} \label{lem:vmut}
Given an integer $k$, fix $b$ such that $\tfrac{k-1}{k}<b<\tfrac{k}{k+1}$, and let $Q_0:=Q_b$ as in Fig.~\ref{fig:basediagr}.
\begin{itemize}
    \item[{\rm(i)}] The quadrilateral $Q_k:=OX_{v_k}V_{v_k}Y$ that is obtained from $Q_0$ by
    performing $k$ $v$-mutations has the following data: 
\begin{align*}
    |OY|&=1-b, \qquad  |OX_{v_k}|=kb-(k-1), \\
    |V_{v_k}Y|&=k-(k-1)b, \qquad |X_{v_k}V_{v_k}|=1-b, \\
     \vn_{Y_{v_k}}&=\begin{pmatrix} 1 \\ -1 \end{pmatrix}, \quad \vn_{V_{v_k}}=\begin{pmatrix} -2k \\ -1 \end{pmatrix}, \quad  \vn_{X_{v_k}}=\begin{pmatrix} 2k-2 \\ 1 \end{pmatrix},  \\
    \ovr{V_{v_k}Y}&=\begin{pmatrix} -1 \\ 0 \end{pmatrix}, \quad \ovr{X_{v_k}V_{v_k}}=\begin{pmatrix} 2k-1 \\ 1 \end{pmatrix}. 
\end{align*}
\item[{\rm (ii)}] One $y$-mutation of $Q_k$ gives $yQ_k:=OX_{v_k}V_{yv_k}Y_y$ with data:
\begin{align*}
     |OY_y|&=1+k-kb, \qquad  |OX_{v_k}|=kb-(k-1), \\
    |V_{yv_k}Y_y|&=\frac{k+(1-k)b}{2k} , \qquad |X_{v_k}V_{yv_k}|=\frac{k-(1+k)b}{2k}, \\
     \vn_{Y_y}&=\begin{pmatrix} 1 \\ -(2+2k) \end{pmatrix}, \quad
    \vn_{V_{yv_k}}=\begin{pmatrix}
    -1 \\ 1
    \end{pmatrix},\quad
    \vn_{X_{v_k}}=\begin{pmatrix} 2k-2 \\ 1 \end{pmatrix},\\
    \ovr{V_{yv_k}Y_y}&=\begin{pmatrix} -1 \\ 1+2k \end{pmatrix}, \quad \ovr{X_{v_k}V_{yv_k}}=\begin{pmatrix} 2k-1 \\ 1 \end{pmatrix}. \\
\end{align*}
\end{itemize}
\end{lemma}
\begin{proof}
For (i), setting $k=0,$ we get the data for $Q_0$. 
We proceed by induction assuming the data holds for $k$ mutations. Note, the nodal ray $\vn_{V_{v_k}}$ will extend to hit the side $|OX_{v_k}|$ if and only if $\frac{k}{k+1}<b$. This holds by assumption as to complete the $(k+1)^{\text{st}}$ mutation, $\frac{k}{k+1}<b$. In performing such a $v$-mutation, it is easily verifiable from the definition that $\vn_{X_{v_{k+1}}}=-\vn_{V_{v_k}}$ and $|OY|$ remains constant. The monodromy matrix $M:=M_{V_{v_k}}$
is \[M=\begin{pmatrix} 1+2k & -4k^2 \\ 1 & 1-2k \end{pmatrix},\] i.e. $M\vn_{V_{v_k}}=\vn_{V_{v_k}}$ and $M\ovr{X_{v_k}V_{v_k}}=\ovr{V_{v_k}Y}$.
Then, we compute 
\begin{align*}
 M\begin{pmatrix} 1 \\ 0 \end{pmatrix}=\begin{pmatrix} 2k+1 \\ 1 \end{pmatrix}=\ovr{X_{v_{k+1}}V_{v_{k+1}}}, \quad \text{and} \quad M\vn_{X_{v_k}}=\begin{pmatrix} -2(1+k) \\ -1 \end{pmatrix}=\vn_{V_{v_{k+1}}}.
\end{align*}
To verify the formula for $|OX_{v_{k+1}}|,$ we use the fact that the vector sum of the sides of the
quadrilateral  $OX_{v_{k+1}}V_{v_k}Y$ is zero, which implies that for suitable $r>0$ we have
\[
(|OX_{v_{k+1}}|,0)+r(2k, 1)+(-|V_{v_k}Y|,0)+(0,- |OY|)=0.
\] 
Thus $r = |OY| = 1-b$, and since $|V_{v_k}Y|= k-(k-1)b$ we obtain 
$|OX_{v_{k+1}}|=(k+1)b-k$ as desired. Furthermore, as $M$ preserves the affine lengths, 
\begin{align*} |X_{v_{k+1}}V_{v_{k+1}}| & =|OX_{v_{k}}|-|OX_{v_{k+1}}|=1-b, \\ 
 |V_{v_{k+1}} Y|  & =  |V_{v_k}Y|+|X_{v_k}V_{v_k}|=(k+1)-kb.
 \end{align*}
This completes (i). 

For (ii), the assumption that $b<\frac{k}{k+1}$ implies that the nodal ray $\vn_{Y_{v_k}}$ will extend to hit the $X_{v_k}V_{v_k}$ edge of $Q_k$. The formulas in (ii) follow easily by computing the effect of one $y$-mutation on $Q_k$  using the data in (i). Details are left to the reader.
\end{proof}

  \begin{center}
 \begin{figure}
 \begin{overpic}[scale=1,unit=0.5mm]{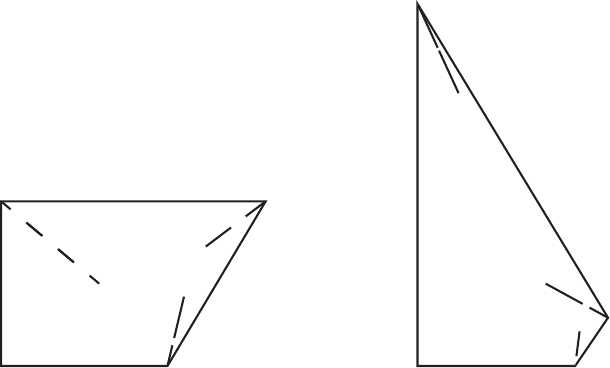}
 	\put (-6,-6) {$O$}
	\put (58,-6) {$X$}
	\put (93,55) {$V$}
	\put (-5,55) {$Y$} 
 	\put (12,-6) {$kb-(k-1)$}
	\put (-20,20) {$1-b$}
	\put (15,60) {$k-(k-1)b$}
	\put (80,20) {$1-b$}
	\put (23,10) {$(2k-2,1)$} 
	\put (40,35) {$(-2k,-1)$}
	\put (75,10) {$(2k-1,1)$}
	\put (15,45) {$(1,-1)$}
	\put (155,30) {$(-1,1)$}
	\put (100,95) {$(1,-2-2k)$}
	\put (180,70) {$(-1,1+2k)$}
  \end{overpic}
  \vspace{5mm}
     \caption{To visualize the approximate proportions of Lemma~\ref{lem:vmut}, on the left is $v^kH_b$ and on the right is $yv^kH_b$. Note, depending on $b$, the nodal ray at $Y$ pointing in the $(1,-1)$ direction of $v^kH_b$ could extend to intersect $OX$ or $XV$. Therefore, in performing the mutation of $y$, we assume $k/(k+1) \leq b,$ to ensure $(1,-1)$ will intersect $XV$.}
     \label{fig:vyMut}
 \end{figure}
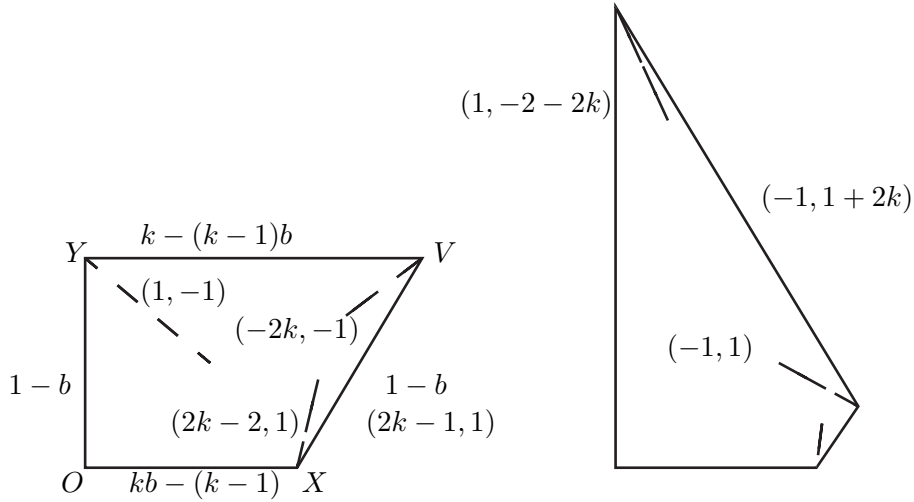
 \end{center} 

\begin{remark} \rm As pointed out by Tara Holm, by Lemma~\ref{lem:vmut} and Figure~\ref{fig:vyMut}, assuming $ \frac{k}{k+1}<b$, $v^kQ_b$ is a Delzant polygon, and hence corresponds to a toric action on $H_b$. This does not hold for the other mutations considered. 
\end{remark} 

  \section{The Result of Mutations}
The goal of this section is to compute the data of the decorated quadrilateral $wy^kv^{n+2}Q_b$ where $w$ is a finite word in $x,y$ for appropriate $b$-values. The main result of this section is Proposition~\ref{prop:triple}, which shows the quadrilaterals defined in Definition~\ref{def:QT} correspond to $wy^kv^{n+2}H_b$ with appropriate restrictions on $b.$

We begin with some basic observations about the mutations we perform. For ease of notation, we write all vectors in the statement of the lemma as row vectors rather than column vectors. 
 \begin{lemma} \label{lem:tMot} \begin{itemlist}
     \item[{\rm (i)}] Assume $Q$ is such that \[\vn_V=(q_0,-p_0), \quad \vn_Y=( q_1, -p_1 ), \quad \ovr{VY}=( 
     -q_1^2 , p_1q_1-1 ),\] and if 
     $ ( q_2 , -p_2):=M_Y\vn_V$, then  $M_Y\ovr{XV}=( -q_2^2 , p_2q_2-1).$ If $b$ is such that $yQ$ is defined, then we have
     \[ M_{Y_y}\vn_{V_y}=(q_3,-p_3), \quad \text{and} \quad 
     M_{Y_y}\ovr{XV_y}=(-q_3^2,p_3q_3-1),
     \] 
     and further, for $r=p,q$ and $i=1,2$, we have \[r_{i+1}=tr_i-r_{i-1}\]
     with recursion parameter $t=:p_1q_0-p_0q_1$.

    \item[{\rm (ii)}] Assume $Q$ is such that 
    \[\vn_V=-( p_0-6q_0 , q_0 ), \quad \vn_X=( p_1-6q_1 , q_1 ),  \quad \ovr{XV}=(
        1+p_1q_1-6q_1^2 , q_1^2 ),\] and if $( p_2-6q_2 , q_2):=M_X\vn_{V},$ then $M_X\ovr{VY}=(
        1+p_2q_2-6q_2^2 , q_2^2 )$. If $b$ is such that $xQ$ is defined, then 
     \[ M_{X_x}\ovr{V_yY}=(
        1+p_3q_3-6q_3^2 , q_3^2 ), \quad \text{and} \quad M_{X_X}\vn_{V_x}=( p_3-6q_3 , q_3 ),
     \] 
     and further, for $r=p,q$ and $i=1,2$, we have \[r_{i+1}=tr_i-r_{i-1}\]
    with recursion parameter $t=:p_0q_1-p_1q_0$. 
        \end{itemlist}
 \end{lemma}
 \begin{proof}
     Given the data of $\vn_Y$ and $\ovr{VY}$, we compute $M_Y=\begin{pmatrix}
         1-p_1q_1 & -q_1^2 \\ p_1^2 & 1+p_1q_1.
     \end{pmatrix}$
     Then, the statement that $(p_2,q_2)$ is determined by the recursive sequence follows by direct computation as
   \[ M_Y\begin{pmatrix} -q_0 \\ p_0 \end{pmatrix}=\begin{pmatrix} tq_1-q_0 \\ -(tp_1-p_0) \end{pmatrix}\]
   for $t:=p_1q_0-p_0q_1.$
   Then, the assumption about $M_Y(\ovr{XV})$ implies 
   \[ M_{Y_y}=\begin{pmatrix}
         1-p_2q_2 & -q_2^2 \\ p_2^2 & 1+p_2q_2
     \end{pmatrix}.\]
    Hence, the calculations above imply the statement about $M_{Y_y}\vn_{Y_y}$ holds (recall, 
   by definition of the $y$-mutation, $\vn_{V_y}=-\vn_Y)$. In doing a $y$-mutation $\ovr{XV_y}=\ovr{XV},$ therefore, the 
   statement about $M_{Y_y}\ovr{XV}$ can also be computed directly. 
For (ii), similar computations can be made. Alternatively, one can observe that the shear $S=\begin{pmatrix} 1 & 6 \\ 0 & 1 \end{pmatrix}$ takes
 \[S(\vn_X)=\begin{pmatrix}
     p \\ q
 \end{pmatrix}, \quad S(\ovr{XV})=\begin{pmatrix} 1+pq \\ q^2 \end{pmatrix}, \quad S\begin{pmatrix} 1 \\ 0 \end{pmatrix}=\begin{pmatrix} 1 \\ 0 \end{pmatrix}
 \]
 and use the results of (i). 
 \end{proof}

 The recursion parameter seen in Lemma~\ref{lem:tMot} has the following property:
  \begin{lemma}\label{lem:recurStay} 
    Let $t$ be an integer. For $k \geq 0$, given sequences $p_k,q_k$ where for some integer $t,$ the values $x=p,q$ are determined by
    \[x_{k+1}=tx_k-x_{k-1}, \] then 
    \[|p_kq_{k-1}-p_{k-1}q_k|=|p_1q_0-p_0q_1|.\]
\end{lemma}
\begin{proof} This is a direct computation:
    \[ -q_kp_{k-1}+p_kq_{k-1}=-(tq_{k-1}-q_{k-2})p_{k-1}+(tp_{k-1}-p_{k-2})q_{k-1}= -p_{k-2}q_{k-1}+p_{k-1}q_{k-2}.\]
\end{proof}

 Assuming $M_Y\ovr{XV}$ and $M_X\ovr{VY}$ satisfy the conditions of Lemma~\ref{lem:tMot}, the nodal rays and directions vectors of sequences of $x$ and $y$ mutations will be determined by recursive sequences given by recursion parameters of the form $|p_iq_{i-1}-p_{i-1}q_i|.$

\begin{example} \label{eg:genTrip} \rm 
Choose $b$ appropriately such that the mutations
$y^kv^2Q_b=:OXV_kY_k$ are well defined. By the computation of $yv^2Q_b$ in Lemma~\ref{lem:vmut}, we can check \[ \vn_{V_1}=\begin{pmatrix}
    -1 \\ 1
\end{pmatrix}, \quad \vn_{Y_1}=\begin{pmatrix}
    1 \\ -6
\end{pmatrix}, \quad \ovr{V_1Y_1}=\begin{pmatrix} -1 \\ 5 \end{pmatrix}, \quad \ \ovr{XY_1}=\begin{pmatrix} 3 \\ 1 \end{pmatrix}\] satisfy the assumptions of Lemma~\ref{lem:tMot}. Hence, by Lemma~\ref{lem:tMot}, we have
\[ \vn_X=\begin{pmatrix}
    2 \\ 1 
\end{pmatrix}, \quad \vn_{V_k}=\begin{pmatrix}
    -q_{k-1} \\ p_{k-1}
\end{pmatrix}, \quad \vn_{Y_k}=\begin{pmatrix} q_{k} \\ -p_{k} \end{pmatrix}, \quad
\ovr{V_kY_k}=\begin{pmatrix} -q_k^2 \\ p_kq_k-1 \end{pmatrix}, 
 \]
where for $r=p,q$, $r_{k+1}=5r_k-r_{k-1}$
with initial conditions given by $(p_1,q_1)=(6,1)$, $(p_0,q_0)=(1,1),$ and the recursion parameter is $p_1q_0-p_0q_1=5.$ Now, we consider performing an $x$-mutation to $y^kv^2Q_b$. The quadrilateral $y^kv^2Q_b$ has \[\vn_X=\begin{pmatrix}
    2 \\ 1 
\end{pmatrix}=\begin{pmatrix}
    \tilde{p}_1-6 \cdot \tilde{q}_1 \\ \tilde{q}_1
\end{pmatrix} \quad \text{for} \quad
(\tilde{p}_1,\tilde{q}_1)=(8,1).\]
The conditions in Lemma~\ref{lem:tMot}~(ii) about $\vn_{V_k}$ and $\ovr{V_kY_k}$ do not hold, but as we see later in Proposition~\ref{prop:triple}, the conditions of Lemma~\ref{lem:tMot}~(ii) do hold after performing one $x$-mutation. Therefore, the quadrilateral  $x^jy^kv^2Q_b:=OX_jV_{k,j}Y_k$ has nodal ray
\[ \vn_{X_j}=\begin{pmatrix} \tilde{p}_{k,j}-6\tilde{q}_{k,j} \\ \tilde{q}_{k,j} \end{pmatrix},
\]
where $(\tilde{p}_{k,0},\tilde{q}_{k,0})=(8,1),$ and for $\tilde r=\tilde p,\tilde q,$ we have 
\[\tilde r_{k,j+1}=t_{k}\tilde r_{k,j}-\tilde r_{k,j-1} , \quad \text{for} \quad t_k:=\tilde{p}_{k,0}\tilde{q}_{k,1}-\tilde{p}_{k,1}\tilde{q}_{k,0}. \]

For the conditions of Lemma~\ref{lem:tMot} (i) and (ii) to hold after interchanging between $x$ and $y$ mutations, the sequences $(p_k,q_k)$ and $(\tilde{p}_{k,j},\tilde{q}_{k,j})$ satisfy certain compatibility conditions. These compatibility conditions are outlined in Section~\ref{sec:compat}. 

\hfill$\er$
\end{example}

The affine lengths are also determined by  sequences of the same recursion parameters as the nodal rays with different initial conditions. As such, we consider tuples of four parameters $(d,m,p,q).$ We now define a mutation process on triples\footnote{This perspective of mutation on triples was first considered in \cite{MMW}.} of such tuples, which will correspond to the various recursive sequences depending on if we mutate by $x$ or $y$ as seen in Example~\ref{eg:genTrip}.

  \begin{definition} \label{def:mutTrip} For $\bullet=\la,\mu,\rho$, define the tuples $\bE_\bullet:=(d_\bullet,m_\bullet,p_\bullet,q_\bullet) \in \Z_{\geq 0}^4.$ If $\Tt:=(\bE_\la,\bE_\mu,\bE_\rho)$ is a triple of tuples, then we define the {\bf $x$-mutation of $\Tt$} to be the triple $x\Tt$ and the {\bf $y$-mutation of $\Tt$} to be the triple $y\Tt$ such that $x\Tt:=(\bE_\la,\bE_{x\mu},\bE_{\mu})$ and $y\Tt:=(\bE_\mu,\bE_{y\mu},\bE_\rho)$ where 
  \begin{equation}
    \label{eq:muttrip}
   \bE_{x\mu}:=t_\la \bE_\mu-\bE_\rho \quad \text{and} \quad \bE_{y\mu}:=t_\rho \bE_\mu-\bE_\la.
   \end{equation}
and subtraction and multiplication in \eqref{eq:muttrip} are performed term by term and \[t_\la:=|p_\rho q_\mu-p_\mu q_\rho|, \quad \text{and} \quad t_\rho:=|p_\mu q_\la-p_\la q_\mu|.\]
     \end{definition}
     Here the letters $\la,\mu, \rho$ stand for \lq left', \lq middle', and \lq right'. Lemma~\ref{lem:recurStay} implies that $t_\la$ and $t_\rho$ defined in Definition~\ref{def:mutTrip} are well defined with respect to $x$ and $y$ mutations. There is an abuse of notation as we have two notions of mutation with the same notation. If $w=x,y$, given a decorated quadrilateral $Q,$ $xQ$ refers to the mutation on $Q$ as defined in Definition~\ref{def:mutation}, while for a triple $\Tt$, $w\Tt$ refers to the mutation defined above. The motivation for this abuse of notation is to be seen in Proposition~\ref{prop:triple}. 

The triples we now define will determine $wyv^{n+2}Q_b$ where $w$ is some finite word of $x$ and $y.$ From considering the perspective of obstructions to symplectic embedded, these triples are also defined in \cite{MMW}. See Section~\ref{sec:obs} for more detail. 

 \begin{definition} \label{def:triples}
    Define the basic triples\footnote{The notation for the basic triples comes from continued fractions. For a tuple $(d,m,p,q),$ if the continued fraction of $p/q=[a_0,\hdots,a_k],$ then $\bE_{[a_0,\hdots,a_k]}$ denotes $(d,m,p,q).$ The relevance of continued fractions is seen in Section~\ref{sec:obs}.}: \[\Tt^*_n:=\left(\bE_{[2n+6]},\bE_{[2n+7,2n+4]},\bE_{[2n+8]}\right)\] where
    \begin{align*} 
    \bE_{[2n+6]}&:=(n+3,n+2,2n+6,1), \\
\bE_{[2n+7,2n+4]}&:=\left(2n^2+11n+14,2n^2+9n+9,4n^2+22n+29,2n+4\right).
    \end{align*} 
Additionally, we define the set of all mutations of the basic triples: 
 \[\Cc_n:=\{ \Tt \ | \ \Tt=w\Tt^*_n \ \text{where $w$ is a finite word of $x$ and $y$} \}.\]
\end{definition}

\begin{example} \rm For $n=0$, the basic triple is 
   \[ \Tt^*_0=\left((3,2,6,1),(14,9,29,4),(4,3,8,1)\right).\]
   Then, to compute $x\Tt^*_0$, we take
   \[ \bE_{x\mu}=|29\cdot 1-4\cdot 8|(14,9,29,4)-(4,3,8,1)=(38,24,79,11),\] and to compute $y\Tt^*_0,$ we take
   \[ \bE_{y\mu}=|6\cdot 4-29\cdot 1|(14,9,29,4)-(3,2,6,1)=(67,43,139,19).\] 
   Note, the numerics of consecutive $y$-mutations here correspond to the numerics in Example~\ref{eg:genTrip}. Using the notation of Example~\ref{eg:genTrip}, we have $y^kv^2Q_b:=OXV_kY_k.$ The triple $(3,2,6,1)$ has $p_1/q_1=6/1$ determining the slope of $\vn_{Y_1}$. 
   The recursion parameter is $|6\cdot 4-29\cdot 1|=5.$ Then, the triples $(14,9,29,4)$ and $(67,43,139,19)$ have
   $29/4=p_2/q_2$ and $139/19=p_3/q_3$ determining the slopes of $\vn_{Y_2}$ and $\vn_{Y_3}.$
   \hfill$\er$
   
\end{example}

\begin{remark} \label{rmk:unique} \rm By \cite[Prop2.2.2]{MMW}, each triple $\Tt \in \Cc_n$ besides $\Tt^*_n$ has a unique predecessor. In particular, for $\Tt \neq \Tt^*_n,$ there exists a unique $\Tt'$ such that $x\Tt'=\Tt$ or $y\Tt'=\Tt.$ Therefore, the nodal ray at $V$ in Definition~\ref{def:QT} is well defined. 
\end{remark}

We now define a decorated quadrilateral for each of the triples in $\Cc_n.$ As noted in Remark~\ref{rmk:unique}, each triple $\Tt$ has a unique predecessor $\Tt'$ such that either $\Tt=x\Tt'$ or $\Tt=y\Tt'$. Then, we refer to this predecessor as either $\Tt'=\ov{x}\Tt:=(\bE_\la,\bE_\rho,\bE_{\oxr})$ (resp. $\Tt'=\ov{y}\Tt:=(\bE_{\oyl},\bE_\la,\bE_\rho)$) if $\Tt=x\Tt'$ (resp. $\Tt=y\Tt'$). For ease of notation, given a tuple $(d,m,p,q)$, denote
\begin{equation}
    d'=d-3q, \quad \text{and} \quad m'=m-q. \label{eq:md'}
\end{equation}
\begin{definition} \label{def:QT}
Let $\Tt=(\bE_\la,\bE_\mu,\bE_\rho) \in \Cc_n$. Then, define the decorated quadrilateral $Q(\Tt):=OXVY$ to be given by the data: 
\begin{align} \label{eq:QT}
    |OY|&=\frac{d_\la-m_\la b}{q_\la}, \quad |OX|=\frac{m_\rho'b-d_\rho'}{q_\rho}, \\ \notag
    |VY|&=\frac{m_\rho'-d_\rho'b}{q_\la q_\mu}, \quad
    |XV|=\frac{m_\la-d_\la b}{q_\rho q_\mu}, \\ \notag
\vn_Y=\begin{pmatrix}
q_\la \\ -p_\la
\end{pmatrix}, \quad \vn_{\ovy V}&=\begin{pmatrix} -q_\oyl \\ p_\oyl
\end{pmatrix},
\quad \vn_X=\begin{pmatrix}
p_\rho-6q_\rho \\ q_\rho
\end{pmatrix},  \quad \vn_{\ovx V}=-\begin{pmatrix} p_{\oxr}-6q_{\oxr} \\ q_{\oxr}
\end{pmatrix}, \\ 
\notag
\ovr{OY}&=\begin{pmatrix}
0 \\ 1
\end{pmatrix}, \quad \quad \quad
\ovr{OX}=\begin{pmatrix}
1 \\ 0
\end{pmatrix},\\
\notag
\ovr{VY}&=\begin{pmatrix}
-q_\la^2 \\ q_\la p_\la-1
\end{pmatrix}, \quad
\ovr{XV}=\begin{pmatrix}
1+p_\rho q_\rho-6q_\rho^2 \\ q_\rho^2
\end{pmatrix},
\end{align} 
where $\vn_{\ovx V}$ (resp. $\vn_{\ovy V}$) is the nodal ray emanating from vertex $V$ if there is a $\Tt'$ such that $x\Tt'=\Tt$ (resp. $y\Tt'=\Tt.$)  For $\Tt=\Tt^*_n,$ we set $\vn_V=(-1,1)^T$.
Here $b\in (0,1)$ is a variable chosen so that all lengths are positive.  Thus we require
\begin{align}\label{eq:bsize}
d'_\rho/m'_\rho < b < \min\bigl(m_\la/d_\la, \ m'_\rho/d'_\rho\bigr),
\end{align}
where $m',d'$ are defined in terms of $m,d,q$  as in \eqref{eq:md'}.
\end{definition}

\begin{prop} \label{prop:triple}
\begin{itemlist}
 \item[{\rm (i)}]   Assuming $\frac{n+1}{n+2} \leq b \leq \frac{n+2}{n+3},$ then we have $Q(\Tt^*_n)=yv^{n+2}Q_b.$  

\item[{\rm (ii)}] Let $\Tt:=(\bE_\la,\bE_\mu,\bE_\rho) \in \Cc_n.$ Assume $b<m_\mu/d_\mu$, then given the decorated quadrilateral $Q(\Tt),$ the $x$ and $y$ mutations of the quadrilateral correspond to $x$ and $y$ mutations on the triples. In other words, 
\[ xQ(\Tt)=Q(x\Tt) \quad \text{and} \quad yQ(\Tt)=Q(y\Tt).\]
\end{itemlist}
\end{prop}
\begin{proof}
For (i), we consider $Q(\Tt^*_n).$ As defined in Definition~\ref{def:QT}, we set $\vn_V=(-1,1)^T.$ The other relevant numbers from $\Tt^*_n$ to determine $Q(\Tt^*_n)$ are:
    \begin{align*} (d_\la,m_\la,p_\la,q_\la)&=(n+3,n+2,2n+6,1) \\
    (d_\rho,m_\rho,p_\rho,q_\rho)&=(n+4,n+3,2n+8,1) \\
    q_\mu&=2n+4. 
    \end{align*} 
    Then, (i) follows from from comparing the data of Lemma~\ref{lem:vmut} for $k=n+2$ to the definition of $Q(\Tt^*_n).$
 
 For (ii),  we delay the necessary computations until Section~\ref{sec:computations}. It follows from Lemma~\ref{lem:xVec} and Lemma~\ref{lem:xnod} that $xQ(\Tt)=Q(x\Tt)$. These lemmas have no restrictions on $b$ as the nodal ray $\vn_X$ has positive slope and, hence, for any $b$ such that $Q(\Tt)$ is defined, $\vn_X$ extends to hit the side $\ovr{VY}.$ It follows from Lemma~\ref{lem:yVec} and Lemma~\ref{lem:ynod} that if $Q(\Tt)$ is well defined and $b<m_\mu/d_\mu$, then $yQ(\Tt)$ is well defined and further, $yQ(\Tt)=Q(y\Tt).$
\end{proof}

\begin{rmk} \rm \label{rmk:vmutation}
As noted in \cite[Rmk.2.1.14]{MMW}, we can define triples \[\Tt^n_{\ell,seed}:=\left((1,1,1,1,2),\bE_{[2n+6]},\bE_{[2n+8]}\right):=(\bE_\la,\bE_\mu,\bE_\rho).\] By \cite[Lem.3.2.1]{MM}, the class $(1,1,1,1,2)$ is called the seed class as 
 \[t_\rho \bE_\mu-\bE_\la=\bE_{[2n+7,2n+4]},\]
 so it serves as the class before $\bE_{[2n+6]}$ and $\bE_{[2n+7,2n+4]}$ in the recursive sequence. Hence, $y\Tt^n_{\ell,seed}=\Tt^*_{n}$. Furthermore, $x\Tt^n_{\ell,seed}=\Tt^{n-1}_{\ell,seed}.$ There is no $x$ or $y$ mutations bringing 
\[ \Tt^n_{\ell,seed} \mapsto \Tt^{n+1}_{\ell,seed}.\]
The $v$-mutation can be thought of as this mutation.  Iterations of the $v$-mutation can be thought of as moving from tuples $(d,m,p,q)$ with $p/q \in [6,8]$ to those with $p/q$ in $[2n+6,2n+8]$. From this perspective, in Lemma~\ref{lem:vmut}, we show that $Q(\Tt^n_{\ell,seed})=v^{n+2}Q_b$. Once we apply one more $y$-mutation to $v^{n+2}Q_b$, we get the $Q(\Tt^*_n)$ as $y\Tt^n_{\ell,seed}=\Tt^*_n.$ 
Furthermore, it turns out that for $yv^{n+2}Q_b$, $\vn_V = (-1,1)=(-q,p)=\vn_{\ovy V}$ corresponds to viewing $\Tt^*_n$ as the $y-$mutation of the quadrilateral $v^{n+2}Q_b$, which is associated to $\Tt^n_{\ell,seed},$ so this is consistent with Definition~\ref{def:QT} (to prove this association, compare Lemma \ref{lem:vmut} and Definition \ref{def:QT}). 
\end{rmk}

 \section{Properties of the triples in $\Cc_n$} \label{sec:compat} This section will state and prove various compatibility conditions about the triples in $\Cc_n$ in order to complete the computations for the proof of Proposition~\ref{prop:triple}~(ii) in Section~\ref{sec:computations}.
 
  In particular, as will be shown in Lemma~\ref{lem:compatCond} 
each tuple defined in $\Cc_n$ satisfies the following condition:

\begin{definition}
For a tuple $\bE:=(d,m,p,q) \in \N^4$, we call $\bE$ a {\bf Diophantine tuple} if 
\begin{equation} \label{eq:diophEq} 
    d^2-m^2=pq-1, \quad \text{and} \quad 3d-m=p+q.
    \end{equation} 
\end{definition} 

As first shown in \cite[Sec 2.2]{MM}, given a tuple $(d,m,p,q)$, the Diophantine conditions can be stated in an alternative way. This alternative perspective is especially useful for us as it will show how for a triple $(\bE_\la,\bE_\mu,\bE_\rho),$ the $t$-variable $t_\rho$ defined in \eqref{eq:muttrip} can be determined from the tuple $\bE_\rho.$
Namely, in \eqref{eq:diophEq}, the linear equation can be used to express $m$ as a function of $d,p,q$, and then substituting this into the quadratic equation, we get
\begin{equation}
    \label{eq:quadSurf}8d^2-6d(p+q)+p^2+3pq+q^2-1=0
\end{equation} with solution 
\[ d=\tfrac18\left(3(p+q) \pm \sqrt{p^2-6pq+q^2+8}\right).\] Thus, if we define
\begin{equation} \label{eq:tDef}
 t:=\sqrt{p^2-6pq+q^2+8}, \quad \eps:=\pm 1,
 \end{equation}
 then $(d,m)$ in the tuple can be expressed as
\begin{equation} \label{eq:formdm}
    d:=\tfrac18\left(3(p+q)+\eps t\right), \quad m:=\tfrac18\left((p+q)+3\eps t\right). 
\end{equation}

In other words, following \cite[Sec.3.1]{MM}, given the appropriate choice of $\eps,$ we can think of a Diophantine tuple as an integer point on the quadratic surface $X$ defined by \eqref{eq:quadSurf} with coordinates $(d,p,q)$ or $(p,q,t)$. In $(p,q,t)$ coordinates, we can easily write $X$ in matrix notation: setting
\[ A:=\begin{pmatrix} -1 & 3 & 0 \\ 3 & -1 & 0 \\ 0 & 0 & 1 \end{pmatrix}, \quad \bx:=\begin{pmatrix} p \\ q \\ t \end{pmatrix},\]
we have
$X=\{\bx \ | \ \bx^T A \bx=8\}.$

By \cite[Lem.3.1.2]{MM}, given two initial integral points $(p_0,q_0,t_0)$ and $(p_1,q_1,t_1)$ in $X$, these points can be extended to a sequence of points $(p_k,q_k,t_k) \in X$ for all $k \geq 0$ defined by
\[ (p_{k+1},q_{k+1},t_{k+1})=t(p_k,q_k,t_k)-(p_{k-1},q_{k-1},t_{k-1})\]
where $t>0$ if $(p_0,q_0,t_0)$ and $(p_1,q_1,t_1)$ are $t$-compatible: 

\begin{definition} [\cite{MMW}]
Two tuples $(p_0,q_0,t_0)$ and $(p_1,q_1,t_1)$ are called $t$-compatible if 
    \begin{equation} \label{eq:compatDef}
    t_0t_1-4t=p_0p_1-3(p_0q_1+q_0p_1)+q_0q_1, \quad \text{i.e} \quad \bx^TA\bx'=4t.
    \end{equation}
\end{definition}

Recall, from Definition~\ref{def:mutTrip}, the recursion parameter for two seeds $(p_0,q_0,t_0)$ and $(p_1,q_1,t_1)$ are of the form $|p_1q_0-p_0q_1|.$ This motivates: 

\begin{lemma}
    If two tuples $(p_0,q_0,t_0)$ and $(p_1,q_1,t_1)$ are $t$-compatible where $t=|p_1q_0-p_0q_1|,$ then after renaming so that $p_0/q_1<p_1/q_1$ (if necessary) the following relation holds: 
    \begin{equation} \label{eq:adjDef}
        (p_0+q_1)(p_1+q_1)-t_0t_1=8p_0q_1.
    \end{equation}
\end{lemma}
\begin{proof}
    Substituting in $t=|p_1q_0-p_0q_1|$ to the left hand side of \eqref{eq:compatDef}, we get the desired result. 
\end{proof}

\begin{definition} [\cite{MMW}]
    Two tuples $(p_0,q_0,t_0)$ and $(p_1,q_1,t_1)$ are {\bf adjacent} if after renaming so that $p_0/q_0<p_1/q_1$ (if necessary) \eqref{eq:adjDef} holds. 
\end{definition}

 We now give two different lemmas about the conditions that the triples in $\Tt \in \Cc_n$ satisfy. These lemmas will be used to prove  the computations in Section~\ref{sec:computations} for
 Proposition~\ref{prop:triple}. The first lemma is shown in \cite{MMW}, and the second lemma we prove assuming the first. 
\begin{lemma} \label{lem:compatCond}
    Any triple $\Tt \in \Cc_n$ satisfies the following conditions:
    \begin{itemlist}
        \item[{\rm (i)}]  All three tuples in $\Tt$ are Diophantine tuples,
        \item[{\rm (ii)}] $\bE_\la,\bE_\mu$ are adjacent,
        \item[{\rm (iii)}] $\bE_\la,\bE_\mu$ are adjacent and $t_\rho$-compatible, i.e. \[t_\rho=\sqrt{p_\rho^2+q_\rho^2-6p_\rho q_\rho+8}=p_\mu q_\la-p_\la q_\mu,\]
        \item[{\rm (iv)}] $\bE_\rho,\bE_\mu$ are adjacent and $t_\la$-compatible, i.e. \[t_\la=\sqrt{p_\la^2+q_\la^2-6p_\la q_\la+8}= p_\rho q_\mu-p_\mu q_\rho,\]
        \item[{\rm (v)}] $t_\la t_\rho-t_\mu=p_\rho q_\la -p_\la q_\rho,$
        \item[{\rm (vi)}] $m_\rho/d_\rho,m_\la/d_\la>m_\mu/d_\mu$.
    \end{itemlist}
\end{lemma}
\begin{proof}
    These are the conditions of a generating triple defined in \cite[Def.2.1.6]{MMW}. The only difference between this statement and the definition of a generating triple is condition (vi) states 
    $\acc(m_\rho/d_\rho),\acc(m_\la/d_\la)>\acc(m_\mu/d_\mu).$ 
     This is equivalent to (iv) in this lemma because in our case, the ratio of $m_\bullet/d_\bullet>1/3,$ and $\acc(b)$ preserves orientation for $b>1/3.$
    By \cite[Prop.2.1.9]{MMW}, if a triple $\Tt$ is a generating triple, then $x\Tt$ and $y\Tt$ are also generating triples. Therefore, to show the lemma holds for all $\Tt \in \Cc_n,$ it suffices to check it holds for the base triples $\Tt^*_n.$ This is done in \cite[Example 2.1.7]{MMW}.
 \end{proof}

\begin{lemma} \label{lem:exident} 
For a triple $(\bE_\la,\bE_\mu,\bE_\rho) \in \Cc_n$ the following identities hold:\footnote{While here we specifically refer to triples in $\Cc_n$, this lemma implies more generally to generating triples as defined in \cite[Def.2.1.6]{MMW}}
\begin{itemize}  \item[{\rm (i)}] $p_\la+q_\la=q_\mu t_\rho-q_\rho t_\mu$ and $7p_\la-q_\la=p_\mu t_\rho-t_\mu p_\rho$
\item[{\rm (ii)}] $p_\rho+q_\rho=p_\mu t_\la-p_\la t_\mu$ and $p_\rho-7q_\rho=q_\la t_\mu-q_\mu t_\la$
\item[{\rm (iii)}] $p_\mu+q_\mu=q_\rho t_\la+p_\la t_\rho$, $7p_\mu-q_\mu=6 p_\la t_\rho+p_\rho t_\la-q_\la t_\rho$, and \newline $7q_\mu-p_\mu=6q_\rho t_\la+q_\la t_\rho-p_\rho t_\la$
\item[{\rm (iv)}] $p_\la(p_\rho-6q_\rho)+q_\la q_\rho=t_\mu$
\item[{\rm (v)}] $q_\la t_\la+q_\rho t_\rho+q_\mu t_\mu=q_\mu t_\la t_\rho$
\item[{\rm (vi)}] $t_\la\begin{pmatrix} 1+p_\mu q_\mu-6q_\mu^2 \\ q_\mu^2 \end{pmatrix}=q_{x\mu}\begin{pmatrix} p_\mu-6q_\mu \\ q_\mu \end{pmatrix}+q_\mu \begin{pmatrix} p_\rho-6q_\rho \\ q_\rho \end{pmatrix}$
\item[{\rm (vii)}] $-t_\rho \begin{pmatrix} -q_\mu^2 \\ q_\mu p_\mu-1 \end{pmatrix}=q_{y\mu}\begin{pmatrix} q_\mu \\ -p_\mu \end{pmatrix}+q_\mu \begin{pmatrix} q_\la \\ -p_\la \end{pmatrix}$
\end{itemize}
\end{lemma}
\begin{proof}
In this proof, we freely assume all of the adjacency and compatibility conditions stated in Lemma~\ref{lem:compatCond}. 
Note that the formulas in (i), (ii) and (iii) are similar, expressing linear combinations of the $p,q$ coordinates of  one element in the triple in terms 
of the $p,q,t$-coordinates of the other two elements.  Their proofs are very similar, and we only carry out the details for (i).

To prove (i), we use directly \eqref{eq:adjDef} as $\bE_\la$ is adjacent to both $\bE_\rho$ and $\bE_\mu$ to get
\[
p_\la+q_\la=\frac{8 p_\la q_\rho+t_\la t_\rho}{p_\rho+q_\rho}=\frac{8 p_\la q_\mu +t_\la t_\mu}{p_\mu+q_\mu}.
\]
Then, by using the two different expressions above for $p_\la+q_\la$, we solve for $p_\la$ and get
\[ 
p_\la=\frac{t_\la(p_\mu t_\rho-p_\rho t_\mu-q_\rho t_\mu+q_\mu t_\rho)}{8(p_\rho q_\mu-p_\mu q_\rho)}=\tfrac{1}{8}(p_\mu t_\rho-p_\rho t_\mu-q_\rho t_\mu+q_\mu t_\rho),
\] 
where  the second equality comes from $t_\la$-compatibility.
Now, we can substitute this expression for $p_\la$ into the identity for the $\bE_\la$ and $\bE_\mu$ adjacency and solve for $q_\la$ to get
\[ 
q_\la=\frac{t_\la(p_\rho t_\mu -7q_\rho t_\mu -p_\mu t_\rho +7q_\mu t_\rho )}{8(p_\rho q_\mu -p_\mu q_\rho) }=\tfrac{1}{8}(t_\mu (p_\rho -7q_\rho )-t_\rho (p_\mu -7q_\mu )).
\]
Then, by considering $p_\la+q_\la$ and $7p_\la-q_\la$ we get the desired identities in (i).
Further details of (ii),(iii) are left to the reader.

For (iv), note, as $\bE_\la$ and $\bE_\rho$ are adjacent, by definition in \eqref{eq:adjDef}, we have
\[ (p_\la+q_\la)(p_\rho+q_\rho)-8p_\la q_\rho=t_\la t_\rho.\]
If we substitute the left hand side equal $t_\la t_\rho$ into $ t_\la t_\rho-t_\mu=p_\rho q_\la-p_\la q_\rho,$ which holds by Lemma~\ref{lem:compatCond}~(v), we obtain the desired equation
\[ t_\mu=p_\la p_\rho-6p_\la q_\rho +q_\la q_\rho.\]

For (v), using $t_\la$- and $t_\rho$-compatibility, we solve for $q_\mu$ and $p_\mu$, and then simplify using the condition $t_\la t_\rho-t_\mu=p_\rho q_\la-p_\la q_\rho$ to get
\begin{equation} \label{eq:pmuqmu}
 q_\mu=\frac{q_\la t_\la+q_\rho t_\rho}{t_\la t_\rho-t_\mu}, \quad p_\mu=\frac{p_\la t_\la+p_\rho t_\rho}{t_\la t_\rho-t_\mu} .
 \end{equation}
 Using this expression for $q_\mu$, we conclude that 
\[
     q_\la t_\la+q_\rho t_\rho+q_\mu t_\mu=q_\mu t_\la t_\rho\]
     as desired.

For (vi), for the second coordinate, we have 
$q_\mu t_\la q_\mu=(q_{x\mu}+q_\rho)q_\mu$, which follows by the recursion as $q_{x\mu}=t_\la q_\mu-q_\rho.$ For the first coordinate, the terms divisible by $6q_\mu$
following by same identity, so it remains to check
\[t_\la(1+p_\mu q_\mu)=q_{x \mu}p_\mu+q_\mu p_\rho=p_\mu(t_\la q_\mu-q_\rho)+q_\mu p_\rho \iff t_\la=q_\mu p_\rho-q_\rho p_\mu.\]
But this holds by $t$-compatibility. 
We leave the details of (vii) to the reader as it follows similarly to (vi). 
\end{proof}

  \section{Proof of Main Theorem} \label{sec:MainProof}

   In this section, we prove Theorem~\ref{thm:Main1}. 
   Assuming we have chosen $b$ as defined in Theorem~\ref{thm:Main1}, then for all $0<\eps<1$, we aim to construct an embedding
\begin{equation} \label{eq:embedGoal} (1-\eps)E(1,\acc(b)) \sembeds V_b(\acc(b)) H_b . \end{equation}

Denote the quadrilateral $OX_kV_kY_k:=y^kwyv_{n+2}Q_b$. As $OX_kV_kY_k$ is scaled to give embeddings into $H_b$, \eqref{eq:embedGoal} is equivalent to
\begin{equation} \label{eq:embedScaled} (1-\eps) E\left(\frac{1}{V_b(\acc(b))},\frac{\acc(b)}{V_b(\acc(b))}\right) \sembeds H_b. \end{equation}

For $\eps=1$, the embedding \eqref{eq:embedScaled} is a full filling, which implies we expect the limit of $\lim_{k \to \infty}OX_kV_kY_k$ to be a triangle. 
Following \cite[Prop. 2.27]{AADT}, for the limit of $\lim_{k \to \infty}OX_kV_kY_k$ to correspond to the embedding \eqref{eq:embedScaled}, the following conditions must hold: 
\begin{itemlist} 
	\item[{\rm (i)}] As $|OX_k|$ is constant in $k$, for all $k \geq 0$, we must have
$|OX_k|=\frac{1}{V_b(\acc(b))}.$ 
	\item[{\rm (ii)}] The limit $\lim_{k \to \infty} |OY_k|=
 \frac{\acc(b)}{V_b(\acc(b))}$. 
	\item[{\rm (iii)}] The limit $\lim_{k \to \infty} |X_kV_k|=0$. 
	\item[{\rm (iii)}] The limits of the nodal rays $\lim_{k \to \infty} \vn_{y_k}$ and $\lim_{k \to \infty} \vn_{v_k}$ have slopes $\pm \acc(b)$.
\end{itemlist}

The next lemma is used to show these properties hold for the quadrilaterals in Definition~\ref{def:QT}. It is a culmination of many results of Bertozzi, et al. found in \cite{ICERM} and generalizations of these results by the author and McDuff found in \cite{MM}. Here, we give the relevant references for the interested reader. The lemma can be visualized in Figure~\ref{fig:lowerBoundClasses}.

 \begin{center}
 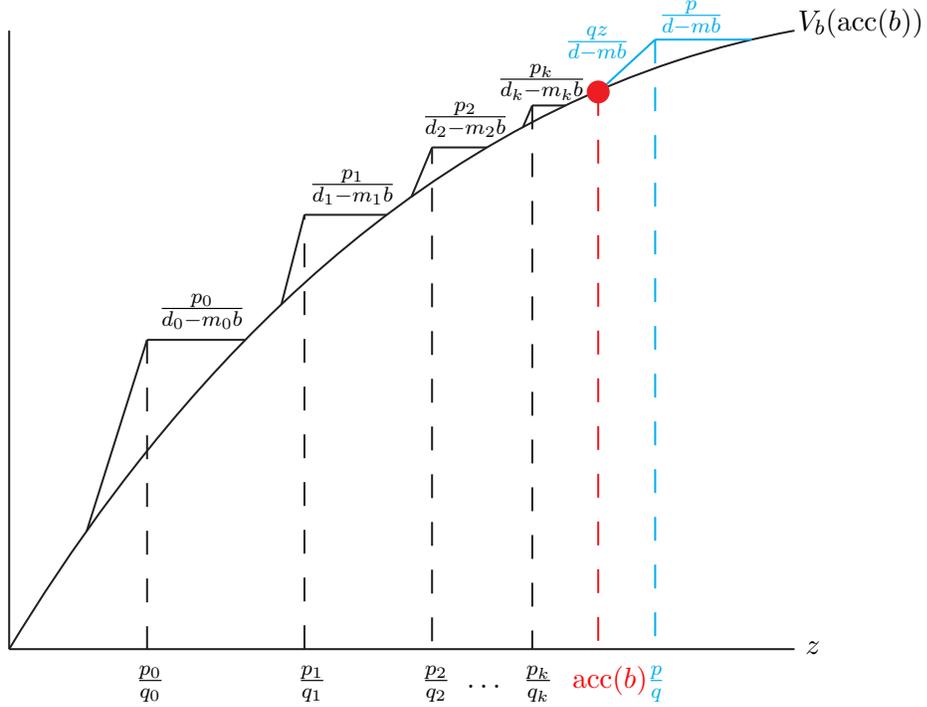
\begin{figure}
 \begin{overpic}[scale=0.75,unit=0.5mm]{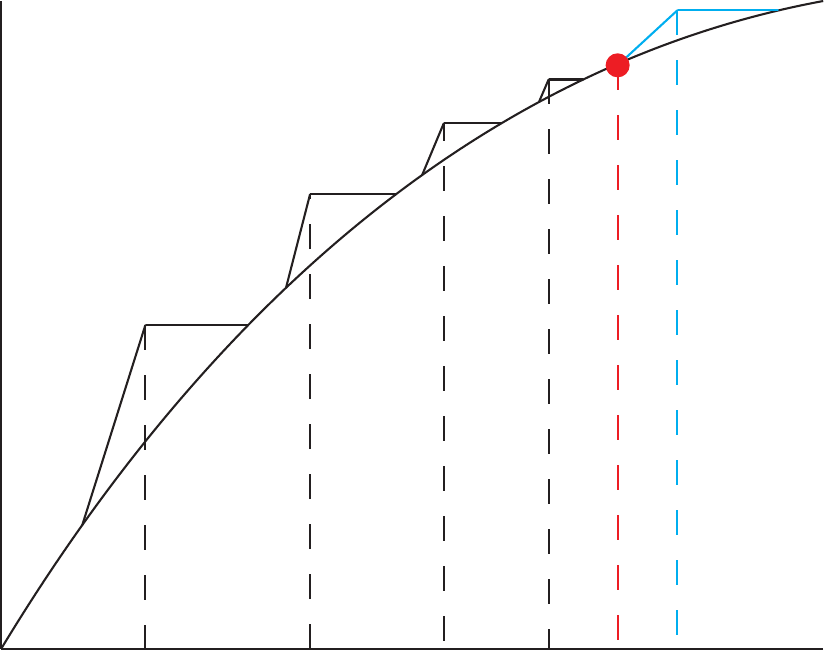}
 	\put (34,-10) {$\frac{p_0}{q_0}$}
	\put (77,-10) {$\frac{p_1}{q_1}$}
	\put (110,-10) {$\frac{p_2}{q_2}$}
	\put (122,-10) {$\hdots$}
	\put (137,-10) {$\frac{p_k}{q_k}$}
	\put (150,-10) {{\color{red} $\acc(b)$}}
	\put (170,-10) {\color{cyan} {$\frac{p}{q}$}}
	\put (40,89) {$\frac{p_0}{d_0-m_0b}$}
	\put (80,122) {$\frac{p_1}{d_1-m_1b}$}
	\put (110,140) {$\frac{p_2}{d_2-m_2b}$}
	\put (130,150) {$\frac{p_k}{d_k-m_kb}$}
	\put (148,160) {{\color{cyan} $\frac{qz}{d-mb}$}}
	\put (173,167) {{\color{cyan} $\frac{p}{d-mb}$}}
	\put (210,165) {$V_b(\acc(b))$}
	\put (212,-1) {$z$}
  \end{overpic}
  \vspace{5 mm}
     \caption{The following cartoon illustrates a lower bound for $c_b(z)$ where there is a sequence of tuples $(d_k,m_k,p_k,q_k)$ along with a tuple $(d,m,p,q)$ satisfying the conditions of Lemma~\ref{lem:LimitMot} where $b:=\acc(\lim_{k \to \infty} \frac{m_k}{d_k}))$.}
     \label{fig:lowerBoundClasses}
 \end{figure}
 \end{center} 

\begin{lemma} \label{lem:LimitMot} Let $\{(d_k,m_k,p_k,q_k)\}_{k \geq 0}$ be a sequence of Diophantine tuples such that 
\[ (d_{k+1},m_{k+1},p_{k+1},q_{k+1})=t(d_k,m_k,p_k,q_k)-(d_{k-1},m_{k-1},p_{k-1},q_{k-1})\]
for $t\geq 3.$ Assume there is a Diophantine tuple $(d,m,p,q)$ such that for all $k$, $(d,m,p,q)$ is adjacent to $(d_k,m_k,p_q,q_k).$
 Denote 
\[  z_\infty:=\lim_{k \to \infty} \frac{p_k}{q_k}, \quad \text{and} \quad b_\infty:=\lim_{k \to \infty} \frac{m_k}{d_k}\]
Then, 
\begin{itemlist}
	\item[{\rm (i)}]  The limits $z_\infty$ and $b_\infty$ exist and have the property that $\acc(b_\infty)=z_\infty$. 
	\item[{\rm (ii)}] $\lim_{k \to \infty} \frac{d_k-m_kb_\infty}{q_k}=\frac{\acc(b_\infty)}{V_{b_\infty}(z_\infty)}$.
	\item[{\rm (iii)}] $V_{b_\infty}(z_\infty)=\frac{qz_\infty}{d-mb_\infty}=\frac{q}{b_\infty(m-q)-(d-3q)}$
\end{itemlist} 
\end{lemma}

\begin{proof}
For (i), \cite[Lemma 3.1.4]{MM} states the limits exist and computes the limits directly in terms of the first two tuples $k=0,1$ by solving the recursion. Then, following \cite[Proposition 41]{ICERM},  if $\lim_{k \to \infty} \frac{p_k}{q_k}=\frac{P}{Q}$ and $\lim_{k \to \infty} \frac{m_k}{d_k}=\frac{M}{D},$ then $\acc(M/D)=P/Q$ exactly if
\[\frac{P+Q}{\sqrt{PQ}}=\frac{3D-M}{\sqrt{D^2-M^2}}.\] By \cite[Lemma 3.1.4]{MM}, this equality is indeed satisfied. This completes (i). 
 
For (ii), as $\lim_{k \to \infty} \frac{p_k}{q_k}=\acc(b_\infty)$, it suffices to show the second equality:
\[ \lim_{k \to \infty} \frac{p_k}{q_k |OY_k|}=\lim_{k \to \infty} \frac{p_k}{d_k-m_kb_\infty}=V_{b_\infty}(z_\infty).\]  By \cite[Lemma 16]{ICERM}, for each tuple $(d_k,m_k,p_k,q_k)$ and any $b$, there is an $\eps>0$ such that for $z \in (\tfrac{p_k}{q_k},\tfrac{p_k}{q_k}+\eps)$, the following lower bounds holds: \[ \mu_{\bE_k,b}(z):=\frac{p_k}{d_k-m_kb} \leq c_b(z).\]

Set $b_k:=\acc^{-1}(\tfrac{m_k}{d_k})$ where we choose the upper inverse if $b_\infty >\frac{1}{3}$ and the lower inverse otherwise. Then, \cite[Proposition 49 (ii)]{ICERM} states that for $z \in (\frac{p_k}{q_k},\frac{p_k}{q_k}+\eps)$, for large enough $k$,
\[ V_{b_k}(z) \leq \frac{p_k}{d_k-m_kb_k}.\] Further, by \cite[Lemma 15 (i)]{ICERM}, the right most inequality also holds
\[ V_{b_k}(z) \leq \frac{p_k}{d_k-m_kb_k}  \leq V_{b_k}(z)\sqrt{1+1/(d_k^2-m_k^2)}.\]
As these are Diophantine classes $d_k^2-m_k^2=p_kq_k-1$, and as $k \to \infty$, $p_kq_k-1 \to \infty.$ By the continuity of the function $(z,b) \mapsto c_b(z)$ and as $b_\infty=\lim_{k \to \infty}b_k$ and $z_\infty=\lim_{k \to \infty} \frac{p_k}{q_k}$, we conclude  
\[ \lim_{k \to \infty} \frac{p_k}{d_k-m_kb_\infty} =V_{b_\infty}(z_\infty)\] 
as desired. 

For (iii), by \cite[Lemma 2.2.7]{MM}, the graph of the function $z \mapsto \frac{1+z}{3-b_\infty}$ passes through the point $(\acc(b_\infty),V_{b_\infty}(\acc(b_\infty)))$. Further, by the assumption that $(d,m,p,q)$ is adjacenct to $(d_k,m_k,p_k,q_k)$, \cite[Theorem 52]{ICERM} implies\footnote{In the language of \cite{ICERM}, the tuple $(d,m,p,q)$ is the blocking class associated to the sequence of tuples.} the function $z \mapsto \frac{qz}{d-mb_\infty}$ also passes through the point $(\acc(b_\infty),V_{b_\infty}(\acc(b_\infty)))$. Hence, 
\[ \frac{1+\acc(b_\infty)}{3-b_\infty}=\frac{qz_\infty}{d-mb_\infty} \]
This implies that
\[ d-mb_\infty=z_\infty(b_\infty(m-q)-(d-3q)).\]
The result follows readily. 
\end{proof}

\begin{remark} \rm \label{rmk:irrb}
The argument in the proof of Theorem~\ref{thm:Main1} is made possible by the fact that for the $b$-values we consider $\frac{1}{V_b(\acc(b))}$ is a linear function of $b$ with rational coefficients; see Lemma~\ref{lem:LimitMot}~(iii).
Thus, a finite number of mutations will result in a figure with $|OX|=1/V_b(\acc(b)).$  The sequence of quadrilaterals $y^kQ$ considered in Theorem~\ref{thm:Main1} do not themselves give rise to  optimal embeddings, although the limit gives the optimal embedding at the accumulation point. Instead, we get upper bounds for the ellipsoid embedding function at a sequence $z_k$ of points converging to $\acc(b)$ that are represented by a sequence of points $\bigl(z_k, V_b(\acc(b))\bigr)$  on the horizontal line at height $ V_b(\acc(b))$.

The paper \cite{M} uses a similar process  for the ellipsoid embedding function with target a two-fold blowup of $\CP^2$ with irrational size blowups. In this case, the reciprocal of the volume is also a linear function of the blowup sizes with rational coefficients, so the ATFs for $H_b$ behave similarly to the ATFs used in \cite{M}. For the two-fold blowup, extra mutations are done to obtain the embeddings that correspond to the inner corners of the ellipsoid embedding function. Experimentally, it seems a similar process is possible for $H_b,$ but in this paper, because we are only using  ATFs to show that $H_b$ is unobstructed, we do not do this extra step to compute the inner corners.

Note that this general process is not possible in the work of \cite{AADT}, which uses ATFs to show that $H_{1/3}$ has an infinite staircase. In this case where $b$ is rational, the volume at the accumulation point is irrational, so $\frac{1}{V_b(\acc(b))}$ is not a rational combination of $1,b$. Rather in \cite{AADT}, each consecutive mutation gave an optimal embedding corresponding to the inner corners of the ellipsoid embedding function.\hfill$\er$
\end{remark}

\begin{definition} \label{def:brho}
    For a triple, $\Tt:=(\bE_0,\bE_1,\bE),$ let $(\bE_k,\bE_{k+1},\bE)$ denote $y^k\Tt.$ Then, define \[b_{\bE}:=\lim_{k \to \infty} \frac{m_{k}}{d_{k}}.\]
    By Lemma~\ref{lem:LimitMot}, this limit exists and further
    \[ \acc(b_\bE)=\lim_{k \to \infty} \frac{p_{k}}{q_{k}}.\]
\end{definition}

\begin{lemma} \label{lem:bDefined} 
    For all $\Tt:=(\bE_\la,\bE_\mu,\bE_\rho) \in \Cc_n$, the value $b_{\bE_\rho}$ has the following properties: 
    \begin{itemlist}
    \item[{\rm (i)}] $\frac{n+1}{n+2} \leq b_{\bE_\rho}\leq \frac{n+2}{n+3}$
    \item[{\rm (ii)}] If $\Tt=w_j\hdots w_1\Tt^*_n$ where for all $1 \leq i \leq j,$ $w_i$ is $x$ or $y,$
    then for all triples $1 \leq i \leq j$, $w_i \hdots w_1 \Tt^*_n=(\bE_{\la,i},\bE_{\mu,i},\bE_{\rho,i}),$
    \[ \frac{m_{\la,i}}{d_{\la,i}},\frac{m_{\mu,i}}{d_{\mu,i}}>b_{\bE_\rho}.
    \]
    \item[{\rm (iii)}] The triples $y^k\Tt=(\bE_{y^{k-1}\mu},\bE_{y^k\mu},\bE_\rho)$ have the property that
    $m_{y^k\mu}/d_{y^k\mu}>b_{\bE_\rho}$ for all $k \geq 0.$
    \end{itemlist}
\end{lemma}
\begin{proof}
    For (i), by Definition~\ref{def:triples}, the left tuple of $\Tt^*_n$ is $(d,m,p,q)=(n+3,n+2,2n+6,1)$. Therefore, by Lemma~\ref{lem:compatCond}~(vi), the value $b_{\bE_\rho}$ determined by $\Tt=(\bE_\la,\bE_\mu,\bE_\rho) \in \Cc_n$
    have the property that $(n+2)/(n+3)>b_{\bE_\rho}$. To see the lower bound on $b_{\bE_\rho},$ we use \cite[Lem.2.3.5]{MM}. The lemma uses the notation $J_{\bB^U_{n+1}}$ to denote an interval contained in $[0,1).$ The left endpoint of this interval is $b_{\bE_{[2n+8]}}$ for the triple $\Tt^*_{n}.$  If we take $x^k\Tt^*_{n+1}=(\bE_{[2n+8]},\bE_{k+1},\bE_k),$ then the right endpoint of $J_{\bB^U_{n+1}}$ is 
    $ \lim_{k \to \infty} \frac{m_k}{d_k}.$ 
    Then, \cite[Lem.2.3.5]{MM} states that for $n \geq 0$ letting $(d_n,m_n):=(n+3,n+2),$ we have $m_n/d_n \in J_{\bB^U_{n+1}}.$
    This implies that for all triples $\Tt':=(\bE_\la',\bE_\mu',\bE_\rho') \in \Cc_{n+1},$ $m_n/d_n<b_{\bE_{\rho}'}$. Therefore, for $n \geq 1,$ we have $(n+1)/(n+2)<b_{\bE_\rho}.$ For the case $n=0,$ we can directly compute the limits defining the endpoints of $J_{\bB^U_0},$ to conclude that the result holds for $n=0.$ This concludes (i).

    The inequalities in (ii) and (iii) follow from Lemma~\ref{lem:compatCond}~(vi) which states that for $\Tt \in \Cc_n,$ $m_\rho/d_\rho,m_\la/d_\la>m_\mu/d_\mu.$
\end{proof}
We now give a proof of the main result Theorem~\ref{thm:Main1}:
\begin{proof}[Proof of Theorem~\ref{thm:Main1}]
    For each word $wyv^{n+2},$ let $\Tt:=(\bE_0,\bE_1,\bE)$ denote $w\Tt_n^*$ where $\bE:=(d,m,p,q)$, and $b_{\bE}$ be as in Definition~\ref{def:brho}. To prove the result, we show that in the notation of the theorem $b(w,n)=b_\bE,$ i.e. 
    \[ \lim_{k \to \infty} y^kwyv^{n+2}Q_{b_\bE}\] is a triangle and further this gives \[c_{b_\bE}(\acc(b_\bE))=V_{b_\bE}(\acc(b_\bE)).\]

    By Proposition~\ref{prop:triple}~(i), if $\frac{n+1}{n+2} \leq b_\bE \leq \frac{n+2}{n+3},$ then the mutations $yv^{n+2}Q_{b_\bE}$ are defined and $yv^{n+2}Q_{b_\bE}= Q(\Tt^*_n).$ The fact that $\frac{n+1}{n+2} \leq b_\bE \leq \frac{n+2}{n+3}$ holds by Lemma~\ref{lem:bDefined}~(i). By Lemma~\ref{lem:bDefined} (ii) and (iii), the condition on $b_\bE$ in Proposition~\ref{prop:triple} holds for each $k \geq 0$ of $y^kwyv^{n+2}Q_{b_\bE}.$ Hence, for each $k \geq 0,$ $y^kwyv^{n+2}Q_{b_\bE}=Q(y^kw\Tt^*_n)$.

    Let 
    \[(\bE_k,\bE_{k+1},\bE):=
y^kwyv^{n+2}\Tt_n^*  \quad \text{and} \quad OXV_kY_k:=y^kwyv^{n+2}Q_{b_\bE}=Q(y^kw\Tt_n^*).\]
By Lemma~\ref{lem:compatCond}~(iv), for each $k$, $\bE_{k+1}$ is adjacent to $\bE$, so Lemma~\ref{lem:LimitMot} holds for $\{\bE_k\}_{k \geq 0}$ with adjacent tuple $(d,m,p,q).$. 
    Then, by Definition~\ref{def:QT} and Lemma~\ref{lem:LimitMot},  we have 
    \begin{itemlist}
        \item[{\rm(i)}] $|OX|=\frac{m'b_\bE-d'}{q}=\frac{1}{V_{b_\bE}(\acc(b_\bE))}$
        \item[{\rm (ii)}]$|XV_k|=\frac{m_k-d_kb_\bE}{qq_{k+1}} \ \text{and} \ b_\bE=\lim_{k \to \infty} m_k/d_k, \ \implies \ \lim_{k \to \infty} |XV_k|=0.$
        \item[{\rm (iii)}] $|OY_k|=\frac{d_k-m_kb}{q_k} \implies \lim_{k \to \infty}|OY_k|=\frac{\acc(b)}{V_b(\acc(b))}$
        \item[{\rm (iv)}] The slopes of the nodal rays $\vn_{Y_k}$ and $\vn_{V_{k+1}}$ are given by $\pm p_k/q_k \to \acc(b_\bE)$ as $k \to \infty.$
    \end{itemlist}
    Therefore, we conclude that
    \[ \lim_{k \to \infty} y^kwyv^{n+2}Q_{b_\bE}\] is a triangle denoted by $OX_\infty Y_\infty$ where 
    $|OX_\infty|=\frac{1}{V_{b_\bE}(\acc(b_\bE)}$ and $|OY_\infty|=\frac{\acc(b_\bE)}{V_{b_\bE}(\acc(b_\bE))}.$ Hence, by \cite[Proposition 2.2.7]{AADT}, for all $0 < \eps < 1$, the embeddings
    \[ (1-\eps)E\left(\frac{1}{V_{b_\bE}},\frac{\acc(b_\bE)}{V_{b_\bE}(\acc(b_\bE))}\right) \sembeds H_{b_\bE} \iff (1-\eps)E(1,\acc(b_\bE)) \sembeds V_{b_\bE}(\acc(b_\bE))H_{b_\bE}\] hold. This embedding implies that $c_{b_\bE}(\acc(b_\bE)) \leq V_{b_\bE}(\acc(b_\bE)),$ and as the volume constraint is also a lower bound, we get the desired equality. 
\end{proof}

   \section{The details of Proposition~\ref{prop:triple}} \label{sec:computations} In this section, we complete the lemmas necessary to prove Proposition~\ref{prop:triple}. These lemmas rely on the identities computed in Lemma~\ref{lem:exident}.
   
   The first two lemmas show that the affine lengths of the sides of $xQ$ and $yQ$ are as expected.
\begin{lemma}\label{lem:xVec}
Assume that $Q=OXVY=Q(\Tt)$ for $\Tt=(\bE_\la,\bE_\mu,\bE_\rho) \in \Cc_n.$ Then the affine length formulas of $xQ(\Tt)$ equal the affine length formulas of $Q(x\Tt)$, that is
\begin{itemize}
    \item[{\rm (i)}] $|OX_x|=\frac{m_\mu'b-d_\mu'}{q_\mu}$\vspace{.07in}
    \item[{\rm (ii)}] $|V_xY|=\frac{m_\mu'-d_\mu'b}{q_\la q_{x \mu}}$\vspace{.07in} 
    \item[{\rm (iii)}] $|X_xV_x|=\frac{m_\la-d_\la b}{q_\mu q_{x\mu}}$\vspace{.07in} 
\end{itemize}
\end{lemma}
\begin{proof} 
Here, we use the fact that the mutation preserves the affine lengths. For (i), note first that $$
|OX_x|  = |OX| + |XV| = \frac1{q_\rho}(m_\rho' b - d_\rho') + \frac 1{q_\rho q_\mu}  (m_\la - d_\la b).
$$
We may check that $|OX_x|= \frac1{q_\mu}(m_\mu' b - d_\mu') $ by considering the constant term and coefficient of $b$ separately.
The constant terms  will match  provided that
\[ \frac{m_\la}{q_\rho q_\mu }-\frac{d_\rho'}{q_\rho }=-\frac{d_\mu'}{q_\mu } \iff m_\la=d_\rho q_\mu -d_\mu q_\rho .\]
If we rewrite this, using the formulas in \eqref{eq:formdm} to write $m_\la,d_\rho,d_\mu$ in terms of $p,q,t$,
 we find that 
\[ m_\la=d_\rho q_\mu -d_\mu q_\rho \iff p_\la+q_\la=q_\mu t_\rho-t_\mu q_\rho,\]
which holds by Lemma~\ref{lem:exident}~(i).
For the $b$-coefficient we need,
 \[ -\frac{d_\la}{q_\rho q_\mu }+\frac{m_\rho' }{q_\rho }=\frac{m_\mu' }{q_\mu } \iff 
 d_\la=m_\rho q_\mu -m_\mu q_\rho.\]
 Again, substituting the centers of the classes to write $d_\la,m_\rho,m_\mu,$ we find
 \[ d_\la=m_\rho q_\mu -m_\mu q_\rho \iff p_\la+q_\la=q_\mu t_\rho-t_\mu q_\rho,\]
 which holds by Lemma~\ref{lem:exident}~(i). This completes the proof of (i).
 
 For (ii), to find the length $|V_xY|$ we use the fact that $V_x$ is the point of intersection of the side $YV$ with the nodal ray 
 $\vn_X = (p_\rho-6q_\rho, q_\rho) $ from $X$.  Thus, denoting  affine distance travelled along $\vn_X$ by $r$, 
 we must solve for $|V_xY|$ given
 \[\bigl(|OX|+(p_\rho-6q_\rho)r, q_\rho r)=(q_\la^2 |V_xY|,|OY|- (-1+p_\la q_\la) |V_xY|\bigr).\]
 \MS
 For the first entry, we get $r=\frac{-q_\la ^2 |V_xY| -|OX|}{p_\rho-6q_\rho}.$
 Substituting this in for $r$ in the equality of the second entry and solving for $|VY|_x,$ we get 
 \begin{align}
    \frac{|OY|(p_\rho-6q_\rho)+|OX|q_\rho}{(-1+p_\la q_\la)(p_\rho-6q_\rho)+q_\la^2 q_\rho}&=|V_xY| \\
    \label{eq:qxmu}
    \frac{|OY|(p_\rho-6q_\rho)+|OX|q_\rho}{p_\rho(-1+p_\la q_\la)+q_\rho(6-6p_\la q_\la+q_\la^2)}&=|V_xY|.
 \end{align}
 
 First, simplifying the denominator, 
   \[ -p_\rho+6q_\rho+q_\la(p_\rho p_\la+q_\la q_\rho-6p_\la q_\rho)=
    -p_\rho+6 q_\rho +q_\la t_\mu=
    t_\la q_\mu-q_\rho 
    =q_{x \mu},\]
    which follows from Lemma~\ref{lem:exident}~(iii)~and~(v).

For the numerator, we have
\[ |OY|(p_\rho-6q_\rho)+|OX|q_\rho=\frac{d_\la p_\rho -d_\rho'q_\la-6q_\la q_\rho+b(-m_\la p_\rho+m_\rho' q_\la+6m_\la q_\rho)}{q_\la}.\]
First, we must show the constant term is $m_\mu'.$ We first change variables in the constant term from the degree coordinates to the centers. Thus, we must show
\[ 3p_\la(p_\rho-6q_\rho)+(p_\rho-6q_\rho)t_\la-q_\la(-3 q_\rho+t_\rho)=8(m_\mu-q_\mu).\] 
We simplify the LHS using Lemma~\ref{lem:exident}~(v)~and~(viii)
\begin{align*}
    3p_\la(p_\rho-6q_\rho)+(p_\rho-6q_\rho)t_\la-q_\la(-3 q_\rho+t_\rho)&=3(-6q_\rho p_\la+p_\la p_\rho+q_\rho q_\la)+(p_\rho-6q_\rho)t_\la-q_\la t_\rho\\
    &=3t_\mu+(p_\rho-6q_\rho)t_\la-q_\la t_\rho \\
    &=3t_\mu-7q_\mu+p_\mu \\
    &=8(m_\mu-q_\mu)
\end{align*}
as desired completing the constant term.

Doing a similar process for the b-coefficient, this reduces to showing
\[ 8(d_\mu-3 q_\mu)=p_\la(p_\rho-6 q_\rho)+3(p_\rho-6q_\rho)t_\la+q_\la(q_\rho-3t_\rho).\]
We can simplify the RHS using Lemma~\ref{lem:exident}~(v)~and~(viii) to get
\begin{align*} 
 p_\la(p_\rho-6 q_\rho)+3(p_\rho-6q_\rho)t_\la+q_\la(q_\rho-3t_\rho)&=t_\mu-3(7q_\mu-p_\mu) \\
 &=8(d_\mu-3q_\mu).
 \end{align*}
This completes (ii).

For (iii), to show $|VY|-|V_xY|=\frac{m_\la-d_\la b}{q_\mu q_{x \mu}}$, we must show
\begin{equation} \label{eq:xxvx} m_\la=\frac{-m_\mu'q_\mu+m_\rho' q_{\mu x}}{q_\la}, \quad \text{and} \quad -d_\la=\frac{d_\mu ' q_\mu-d_\rho ' q_{x\mu}}{q_\la}. \end{equation}
By substituting in the formulas for the degrees in terms of the center, the first equation is equivalent to
\[ q_\la(p_\la+q_\la+3t_\la)=q_\mu(7q_\mu-p_\mu-3t_\mu)+q_{\mu x}(p_\rho-7q_\rho+3t_\rho).\]
If we substitute in $q_{x \mu}=t_\la q_\mu-q_\rho,$
the terms with coefficient $3$ cancel out by Lemma~\ref{lem:exident}~(v). Then, it remains to show
\[ q_\la(p_\la+q_\la)=q_\mu(7q_\mu-p_\mu)+(t_\la q_\mu-q_\rho)(p_\rho-7q_\rho).\]
This holds by substituting in the identities for $7q_\mu-p_\mu$ and $p_\rho-7q_\rho$ in Lemma~\ref{lem:exident}~(ii)~and~(iii). Thus, the first equation in \eqref{eq:xxvx} hold. A similar computation follows to show the second equation in \eqref{eq:xxvx} holds. The details are left to the reader. 
 \end{proof}

\begin{lemma} \label{lem:yVec}
Assume that $Q=OXVY=Q(\Tt)$ for $\Tt=(\bE_\la,\bE_\mu,\bE_\rho) \in \Cc_n$ and $b<m_\mu/d_\mu$. Then the nodal ray from $Y$ hits the side $XV$ of $Q$, and the affine length formulas of $yQ(\Tt)$ 
 equal the affine length formulas of $Q(y\Tt)$, that is
\begin{itemize}
    \item[{\rm (i)}] $|OY_y|=\frac{1}{q_\mu}(d_\mu-m_\mu b)$;
    \item[{\rm (ii)}] $|XV_y|=\frac{m_\mu-d_\mu b}{q_\rho q_{y\mu}}$; 
    \item[{\rm (iii)}] $|V_yY_y|_y=\frac{m_\rho'-d_\rho' b}{q_\mu q_{y\mu}}=|XV|-|XV|_y$.
\end{itemize}
\end{lemma}
\begin{proof}
The proof that the affine length formulas of $yQ$ are as claimed is very similar to Lemma~\ref{lem:xVec}, except that we use
the identities in Lemma~\ref{lem:exident}~(ii) rather than those in Lemma~\ref{lem:exident}~(i).
Note that this argument does not use the fact that the side lengths are positive.  Indeed, as in the proof of Lemma~\ref{lem:xVec}~(ii), we simply define $|XV_y|\in \R$ to be the solution of the equation
$$
\bigl(|OX| + |XV_y| (1+p_\rho q_\rho - 6q_\rho^2), |XV_y| q_\rho^2\bigr) = \bigl(sq_\la, |OY| - s p_\la\bigr)
$$
that expresses the fact that the nodal ray from $Y$ meets the line through  $X$ and $V$ in the point $V_y$.   Thus 
$|XV_y|>0$ precisely if $V_y$ has positive $y$ coordinate.
We then calculate that the $y$-coordinate of $|XV_y|$ is $(m_\mu-d_\mu b)/(q_\rho q_{y\mu})$.  By assumption, we have 
$m_\mu/d_\mu>b$, so $|XV_y|$ is positive.
Thus $V_y$ lies on the side $XV$ as claimed.
Further details are left to the reader. 
\end{proof}

Next we show given that the nodal rays and direction vectors of $xQ$ and $yQ$ are as expected.

\begin{lemma} \label{lem:xnod} For $\Tt \in \Cc_n$,  letting $Q(\Tt)=OXVY$, the nodal rays and direction vectors of $xQ(\Tt)$
satisfy the formulas of the quadrilateral $Q(x\Tt)$, that is
\begin{itemize}
     \item[{\rm (i)}] $M_X =\begin{pmatrix} 1-p_\rho q_\rho+6q_\rho^2 & (p_\rho-6q_\rho)^2 \\ -q_\rho^2 & 1+p_\rho q_\rho-6q_\rho^2 \end{pmatrix}$
     \item[{\rm (ii)}] $M_X\vn_{\ovx V}=\begin{pmatrix}
    p_\mu-6q_\mu \\ q_\mu
    \end{pmatrix}=\vn_{\ovx{V_x}}$ 
    \item[{\rm (iii)}] $M_X\vn_{\ovy V}=\begin{pmatrix}
    p_\mu-6q_\mu \\ q_\mu
    \end{pmatrix}=\vn_{\ovx V_x}$ 
    \item[{\rm (iv)}] $M_X(\ovr{VY})=\begin{pmatrix}
   1+p_\mu q_\mu-6q_\mu^2 \\ q_\mu^2
    \end{pmatrix}=\ovr{X_xV_x}$
\end{itemize}
\end{lemma}
\begin{proof}
For (i), recall the definition of $M_X$ in Definition~\ref{def:Q}. Then, (i) follows by direct computation. For (ii), we can perform the matrix multiplication and simplify each entry by using the $t$-compatibility conditions. For the first entry, we have
\begin{align*}& p_\rho(q_\rho p_{\oxr}-p_\rho q_{\oxr})+6q_\rho(p_\rho q_{\oxr}-q_\rho p_{\oxr})+6q_{\oxr}-p_{\oxr}&\\
&\qquad = p_\rho t_\la-p_{\oxr}+6(q_{\oxr}-q_\rho t_\la)=p_\mu-6q_\mu,
\end{align*} and for the second entry, we have
\[ q_\rho(p_{\oxr}q_\rho-q_{\oxr}p_\rho)-q_{\oxr}=q_\rho t_\la-q_{\oxr}=q_\mu.\]
This completes (ii).

For (iii), we begin with the second entry of $M_X\vn_{\ovy V}:$
\begin{equation} \label{eq:xqmu}
     q_\rho(p_{\oyl}p_\rho+q_\rho q_{\oyl}-6q_\rho p_{\oyl})+p_{\oyl}=q_\rho t_\la+p_{\oyl}=t_\rho q_\la-q_{\oyl}=q_\mu,
\end{equation}
where this follows from Lemma~\ref{lem:exident}~(i)~and~(v), and the recursive structure of the triples. 
For the first entry of $M_{x}\vn_{\ovx y}$, we want to show $ p_\mu-6 q_\mu=p_\oyl(p_\rho-6q_\rho)^2-(1-p_\rho q_\rho+6 q_\rho^2)q_\oyl.$ Using the expression found for $q_\mu$ in \eqref{eq:xqmu}, this is equivalent to showing $$
p_\mu=6p_{\oyl}+p_\rho(p_\rho p_\oyl+q_\rho q_{\oyl}-6p_{\oyl}q_\rho)-q_{\oyl}.
$$
For the RHS, we have
\[6p_{\oyl}+p_\rho(p_\rho p_\oyl+q_\rho q_{\oyl}-6p_{\oyl}q_\rho)-q_{\oyl}=6p_{\oyl}+p_\rho t_\la-q_{\oyl}=t_\rho P_\la-p_\oyl=p_\mu\]
which holds by Lemma~\ref{lem:exident}~(ii)~and~(v). 

This concludes the proof of (iii). 
The proof of (iv) is given in  Lemma~\ref{lem:vy}.
\end{proof}

\begin{center}
  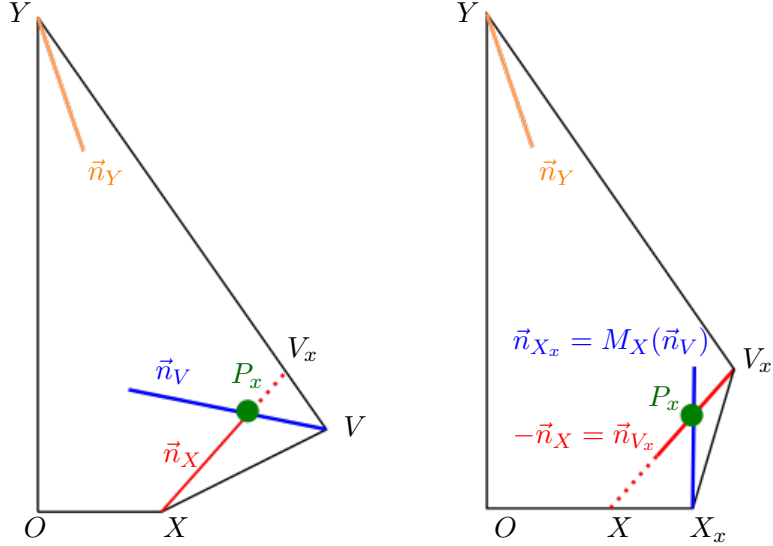
\begin{figure}[ht]
\begin{overpic}[%grid,
scale=0.75,unit=0.5mm]{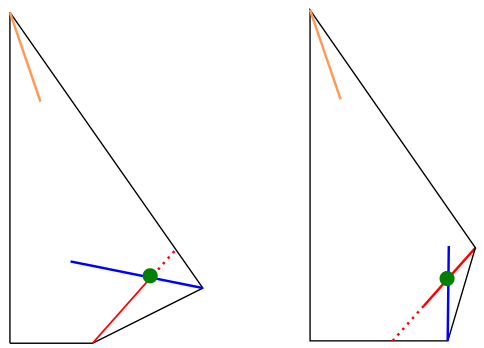}
   %left figure
   \put(0,-5){$O$}
   \put (-4,132){$Y$}
   \put (37,-5){$X$}
   \put (85,22){$V$}
   \put (70,43) {$V_x$}
   \put (37,15) {\color{red}$\vn_X$}
   \put (35,37) {\color{blue}$\vn_V$}
   \put (55,36) {\color{black!60!green}$P_x$}
   \put (17,90) {\color{orange}$\vn_Y$}
   %right figure
   \put (137,90) {\color{orange}$\vn_Y$}
   \put (166,30) {\color{black!60!green}$P_x$}
   \put(125,-5){$O$}
   \put (115,132){$Y$}
   \put (155,-5){$X$}
   \put (177,-5){$X_x$}
   \put (191,40){$V_x$}
   %\put (70,43) {\color{red}$V_x$}
   \put (130,20) {\color{red}$-\vn_X=\vn_{V_x}$}
   \put (130,45) {\color{blue}$\vn_{X_x}=M_X(\vn_V)$}
\end{overpic}
\vskip 0.1in
 \caption{On the left we show $Q=OXVY$ and on the right we see its $x$-mutation $xQ=OX_xV_xY.$ The point $P_x$, which is the intersection of $\vn_X$ and $\vn_V$ is fixed.}
    \label{fig:xmut}
\end{figure}
\end{center}

\begin{rmk}\rm\label{rmk:geom}
Consider the mutation by $x$ that is illustrated in Fig.~\ref{fig:xmut}.
This is effected by a matrix $M_X$ that fixes $X$ and the point $V_x$ where the nodal direction $\vn_X$ meets the side $YV$, and takes $V$ to the new vertex $X_x$ on the $x$-axis.  We did not find a direct proof that $M_X$ takes the line segment $V_xV$ (which is in the direction $\ovr{YV}$) to a line $V_xX_x$ in the correct direction, namely $-\ovr{X_xV_x} = \bigl(1 +p_\mu q_\mu-6q_\mu^2, \ q_\mu^2\bigr)$.  (There was a similar difficulty in establishing what the $y$-mutation does to the vector $\ovr{XV}$.)  Instead we used the fact that the $x$-mutation fixes the point $P_x$ of intersection of the nodal rays $\vn_X, \vn_V$ through $X,V$, since it fixes $\vn_X$ and takes the nodal ray $\vn_V$ to the new nodal ray at the new vertex $X_x$.\footnote{Dusa McDuff suggested using this method to simplify the computations.}  Thus, if we think of the line segment $VV_x$ as the vector sum of the line segments $VP_x, P_xV_x$ in the directions $\vn_V, \vn_X$ we can write $VV_x$ as a linear combination of these two (unit) vectors.   
It turns out the the coefficients of this linear combination can be expressed very simply in terms of the coordinates of the elements in $\Tt$; see Lemma~\ref{lem:vy} and in particular \eqref{eq:PMx}. Moreover, because the vectors $\vn_V, \vn_X$  behave in a transparent way under $x$-mutation we can find the required formula for the direction $\ovr{X_xV_x}$.
\end{rmk}

\begin{lemma} \label{lem:vy} $M_{x}(\ovr{VY})=\begin{pmatrix}
   1+p_\mu q_\mu-6q_\mu^2 \\ q_\mu^2
    \end{pmatrix}$
\end{lemma}
\begin{proof}
Let $\vn_V$ be the nodal ray emanating from $V$. As we've seen, the formula for $\vn_V$ depends on if the previous mutation was $x$ or $y$. Here, we will show that regardless of the previous mutation
\begin{align}\label{eq:PMx}
\ovr{VY}=\tfrac{q_{x \mu}}{t_\la}\vn_V+\tfrac{q_\mu}{t_\la} \vn_X.
\end{align}
See Remark~\ref{rmk:geom}, that explains why such a decomposition might prove useful in the current context.
We will consider the two  cases separately. 

First, assume $\Tt=x\overline{\Tt}$ for some $\overline{\Tt}$. Then, we want to solve for $c_1,c_2$ such that
\[ \ovr{VY}=\begin{pmatrix} -q_\la^2 \\ q_\la p_\la-1 \end{pmatrix}=c_1 \vn_{\ovx V}+c_2 \vn_X=-c_1\begin{pmatrix}p_\oxr-6q_\oxr \\ q_\oxr \end{pmatrix}+c_2 \begin{pmatrix} p_\rho-6q_\rho \\ q_\rho \end{pmatrix} .\]
Denote the components of $\ovr{VY}:=(v_1,v_2)^T.$
Using the second component $-c_1q_\oxr+c_2 q_\rho=v_2$ to solve for $c_2$ and then substituting this into the first component, we get 
\[ \tfrac1{q_\rho}(c_1(p_\rho q_\oxr-p_\oxr q_\rho)+v_2(p_\rho-6q_\rho))=v_1.\]
Then, we can solve for $c_1$, and use this to solve for $c_2$, to obtain
\begin{align} \label{eq:c1c2}
    c_1&=\frac{-q_\rho q_\la^2+(q_\la p_\la-1)(p_\rho-6q_\rho)}{p_\oxr q_\rho-p_\rho q_\oxr}, \\ \notag
 c_2&=\frac{-q_\oxr q_\la^2+(q_\la p_\la-1)(p_\oxr-6q_\oxr)}{p_\oxr q_\rho-p_\rho q_\oxr}.\end{align} 
For both constants, by $t_\la$-compatibility, the denominator $p_\oxr q_\rho-p_\rho q_\oxr=t_\la.$
For $c_1,$ the numerator is precisely the denominator of \eqref{eq:qxmu} which we compute in the next line to be $q_{x\mu}.$ Furthermore, we see the numerator of $c_2$ is the corresponding numerator of $c_1$ for the triple $(\bE_\la,\bE_\rho,\bE_\oxr)$ rather than the triple $(\bE_\la,\bE_\mu,\bE_\rho)$. Thus, we can conclude that 
\[ -q_\oxr q_\la^2+(q_\la p_\la-1)(p_\oxr-6q_\oxr)=q_\mu.\] 
We find $c_1=\frac{q_{x \mu}}{t_\la}$ and $c_2=\frac{q_\mu}{t_\la}$ as desired. 

Now, consider the case $\Tt=y\overline{\Tt}$. Then, we want to solve for $c_1,c_2$ such that
\[ \ovr{VY}=\begin{pmatrix} -q_\la^2 \\ q_\la p_\la-1 \end{pmatrix}=c_1 \vn_{\ovy V}+c_2 \vn_X=c_1\begin{pmatrix}-q_{\oyl} \\ p_\oyl \end{pmatrix}+c_2 \begin{pmatrix} p_\rho-6q_\rho \\ q_\rho \end{pmatrix} .\]
Similarly to above, we can solve for $c_1$ and $c_2$ to get
\[ c_1=\frac{-q_\rho v_1+(p_\rho-6q_\rho)v_2}{p_\rho p_{\oyl}-6p_{\oyl} q_\rho+q_\rho q_{\oyl}}, \quad \text{and} \quad c_2=\frac{p_\oyl v_1+q_\oyl v_2}{p_\rho p_\oyl-6p_\oyl q_\rho+q_\rho q_\oyl}.\]
In both case, the denominator is $t_\la$ by Lemma~\ref{lem:exident}~(v) for the triple $(\bE_\oyl,\bE_\la,\bE_\rho).$
The numerator of $c_1=q_{x\mu}$ as it's the same numerator as $c_1$ in \eqref{eq:c1c2}.
For the numerator of $c_2$, we have
\[q_{\oyl}(p_\la q_\oyl-q_\la p_\oyl)-q_\oyl=q_\oyl t_\rho-q_\oyl=q_\mu.\]
This concludes the second case. 

In both cases, we have
$c_1=\frac{q_{x\mu}}{t_\la}$ and $c_2=\frac{q_\mu}{t_\la}.$ 
Then, to compute $M_X(\ovr{VY}),$ we consider
\begin{align*} M_X(\ovr{VY})&=c_1 M_X(\vn_{V})+c_2 M_X(\vn_X)\\
&=c_1\begin{pmatrix} p_\mu-6q_\mu \\ q_\mu \end{pmatrix}+c_2 \begin{pmatrix} 
p_\rho-6q_\rho \\ q_\rho
\end{pmatrix}=\begin{pmatrix}
1+p_\mu q_\mu-6q_\mu^2 \\ q_\mu^2
\end{pmatrix}
\end{align*}
where the last equality follows from Lemma~\ref{lem:exident}~(vi) and the computations in Lemma~\ref{lem:xnod}~(ii)~and~(iii).
\end{proof}

\begin{lemma} \label{lem:ynod} For $\Tt \in \Cc_n$, letting $Q(\Tt)=OXVY$, then the nodal rays and direction vectors of $yQ(\Tt)$ 
satisfy the formulas of the quadrilateral $Q(y\Tt),$ that is 
\begin{itemize}
     \item[{\rm (i)}] $M_Y=\begin{pmatrix} 1-p_\la q_\la & -q_\la^2 \\ p_\la^2 & 1+p_\la q_\la \end{pmatrix}$
     \item[{\rm (ii)}] $M_Y\vn_{\ovy V}=\begin{pmatrix}
    q_\mu \\ -p_\mu 
    \end{pmatrix}$ 
    \item[{\rm (iii)}]  $M_Y\vn_{\ovx V}=\begin{pmatrix}
    q_\mu \\ -p_\mu 
    \end{pmatrix}$ 
   \item[{\rm (iv)}] $M_Y(\ovr{XV})=\begin{pmatrix}
   -q_{\mu}^2 \\ q_\mu p_\mu -1
    \end{pmatrix}$ 
\end{itemize}
\end{lemma}
\begin{proof}
(i) follows by direct computation; (ii) follows similarly to Lemma~\ref{lem:xnod}~(ii); (iii) follows similarly to Lemma~\ref{lem:xnod}~(iii) but using the  identities in Lemma~\ref{lem:exident}~(iii),~(iv),~and~(v).
For (iv), we follow a similar procedure to Lemma~\ref{lem:vy} to show that 
\[ \ovr{XV}=-\frac{q_{y\mu}}{t_\rho}\vn_V-\frac{q_\mu}{t_\rho}\vn_Y,\]
and then use Lemma~\ref{lem:exident}~(vii) to conclude the result. We leave the details to the reader. 
\end{proof}

  \section{Relation to Obstructions} \label{sec:obs}
   In this section, we explain how Theorem~\ref{thm:Main1} and the triples $\Cc_n$ defined in Definition~\ref{def:triples} relate to obstructions of symplectic embeddings.

     For positive number $z$, the {\bf weight decomposition} $\bw(z):=\left(w_1^{\times \ell_1},w_2^{\times \ell_2},w_3^{\times \ell_3},\hdots\right)$ is obtained by  inductively decomposing a rectangle with side lengths 1 and $z$ into squares as large as possible. Here, $\times \ell_i$ denotes repeating the entry $w_i$ a total of $\ell_i$ times in $\bw(z)$. Specifically, we set $w_0=\max\{1,z\}$ (note, $w_0$ is not in $\bw(z)$), $w_1=\min\{1,z\}$, and $w_{k}=w_{k-2}-\ell_{k-1} w_{k-1}$ where $\ell_{k}=\lfloor \frac{w_{k-1}}{w_{k}} \rfloor$. The multiplicities $\ell_i$ equal the entries of the continued fraction of $z$. 
   
    For rational $z=p/q,$ work of McDuff in \cite{Mell} generalized by Cristofaro-Gardiner in \cite{CG2} show that 
   \[ E(1,z) \sembeds \la H_b \iff \sqcup_i B(w_i) \sembeds \la H_b\]
   where $\bw(z)=(w_1^{\times \ell_1},\hdots,w_n^{\times \ell_n})$ is the weight expansion of $z$.

   If $\sqcup_i B(w_i) \sembeds \la H_b,$ then we can perform a symplectic blowup along the image of each of the balls $B(w_i).$ This results in a symplectic manifold $(X:=\la H_b \#_i \oCP^2_{w_i},\omega),$ an $(n+1)$-blowup of $\CP^2$ where
   \[ PD[\omega]=\la L-b\la E_0-\sum w_i E_i,\]
   where $L$ is the homology class of a line, $E_0$ is the class of exceptional divisor from the blowup $H_b$, and $E_i$ denotes the class of the exceptional divisor from each of the new blowups. Further, the canonical class for this symplectic form is 
   \[ PD(K)=-3L+ \sum_{i=1}^n E_i.\]
    To denote an element $\bE \in H_2(X,\Z)$, we write $\bE$ in coordinates $(d,m,\bbm):=\bE$ with respect to the basis $\{L,E_0,E_1,\hdots,E_n\}.$

   Denote $\Ee_n$ to be the set of classes in $H_2(\CP^2 \#_{n+1} \oCP^2,\Z)$ that can be represented by an embedded symplectic sphere of self intersection $-1$ and $\Ee:=\bigcup_{n \geq 0} \Ee_n$. 
   
   For all elements $\bE:=(d,m,\bbm) \in \Ee$, the fact that the following intersection is nonnegative gives us a lower bound on the size of the target $\la$ in the embedding: 
   \[ (d,m,\bbm) \cdot PD[\omega] \geq 0 \implies  d\la -m \la b -\sum \tilde{m}_i w_i \geq 0 \implies \la \geq \sum \frac{\tilde{m}_iw_i}{d-mb}=:\mu_{\bE,b}(z)\]
   where we pad the end of either $\bbm$ or $\bw(z)$ with zeros if necessary. We call $\mu_{\bE,b}(z)$ an {\bf obstructive function}. If $z$ is irrational, then the weight expansion $\bw(z)$ is infinite, and the definition of the obstructive function can be extended by only considering the first $j$ elements of $\bw(z)$ if $\bbm$ has length $j.$

   The above discussion implies that for each $\bE \in \E$, we have the lower bound
   \[ \mu_{\bE,b}(z) \leq c_b(z). \]
   In fact, the collective work of \cite{CG,li-li,li-liu,MP,Mell} proves that
   \[ c_b(z)=\sup\{V_b(z),\sup_{\bE \in \E} \mu_{\bE,b}(z)\}.\]
Elements of $\Ee$ have Chern number $1$ and self intersection $-1.$ Defining $\tilde{\Ee}_n$ to be classes in $H_2(\CP^2 \#_{n+1} \oCP^2)$ with Chern number $1$ and self intersection $-1$ and $\tilde{\Ee}=\cup_{n \geq 0} \Ee_n$, then work of Hutchings in \cite{Hutch} shows the supremum over the larger set of classes
 \begin{equation} \label{eq:supComp} c_b(z)=\sup\{V_b(z),\sup_{\bE \in \tilde{\E}} \mu_{\bE,b}(z)\}
   \end{equation}
   also holds. For a class $\bE \in \tilde{\E}$, if there is some $b$ and $z$ such that
\[ c_b(z)=\mu_{\bE,b}(z),\] we say $\mu_{\bE,b}(z)$ is {\bf live} at $z.$

    For our purposes, we restrict to a specific subset of classes in $\tilde{\E}$ referred to as {\bf quasi-perfect classes}. These classes are of the form \[ \bE=(d,m,\bbm), \quad \text{where} \quad q\bw(p/q)=\bbm ,\] and hence we denote such classes as $\bE=(d,m,p,q).$ Due to properties of the weight expansion, for a quasi-perfect class, the conditions regarding the self-intersection and the Chern number are equivalent to  \[ d^2-m^2=pq-1 \quad \text{and} \quad 3d-m=p+q,\]
    which are the same conditions defined in \eqref{eq:diophEq}. Therefore, as used in Section~\ref{sec:compat} such a class $(d,m,p,q)$ is a Diophantine tuple. Hence, all tuples in the triples in $\Cc_n$ correspond to these homology classes in $\E$ giving obstructions. As a slight abuse of notation, for any triple $\Tt \in \Cc_n$ such that $\bE$ is one of the tuples in $\Tt,$ we say $\bE \in \Cc_n.$ 

     \begin{example} \rm
        The Diophantine tuple $(3,2,6,1)$ appearing in $\Tt^*_0$ has $p/q=6/1.$ The weight expansion is
        \[ \bw(6)=(1^{\times 6}),\] and this tuple corresponds to the homology class
        \[ 3L-2E_0-\sum_{i=1}^6 E_i.\] \hfill$\er$
    \end{example}

\begin{definition}
    A quasi-perfect class $\bE$ is a {\bf blocking class} if for some $b \in [0,1)$, 
    \[ V_b(\acc(b)) < \mu_{\bE,b}(\acc(b)).\] 
\end{definition}
    By \eqref{eq:accVol}, if $\bE$ is a blocking class for some $b \in [0,1),$ then $b$ is blocked, and hence, $c_b(z)$ does not have an infinite staircase. For each $\bE \in \Cc_n$, we define 
    \[ J_\bE:=\{ b \in [0,1) \ | \ V_b(\acc(b))<\mu_{\bE,b}(\acc(b)) \}.\] By \cite[Prop.2.2.9]{MM}, for each $\bE \in \Cc_n$, the set $J_\bE$ is a nonempty open interval. We denote the left endpoint $\inf(J_\bE):=b^\ell_\bE.$ A corollary of Theorem~\ref{thm:Main1} is as follows: 
    \begin{cor} \label{cor:4Dusa} 
    For all $\bE \in \Cc_n,$ the left endpoint $b^\ell_\bE$ of the interval $J_\bE$ blocked by $\bE$ is unobstructed, that is
    \[ c_{b^\ell_\bE}(\acc(b^\ell_\bE))=V_{b^\ell_\bE}(\acc(b^\ell_\bE)).\]
    \end{cor}
    \begin{proof} First, we recall the notation from the proof of Theorem~\ref{thm:Main1}. For each word $wyv^{n+2}$, let $\Tt:=(\bE_0,\bE_1,\bE):=w\Tt^*_n$ where $\bE:=(d,m,p,q)$ and $b_\bE$ is as in Definition~\ref{def:brho}, i.e.
    \[ b_\bE:=\lim_{k \to \infty} m_k/d_k\]
 where $y^y\Tt=:(\bE_k,\bE_{k+1},\bE).$ Then, in the proof of Theorem~\ref{thm:Main1}, we showed that 
 \[ \lim_{k \to \infty} y^kwyv^{n+2}Q_{b_\bE}\] is a triangle and 
 \begin{equation} \label{eq:embedMain} c_{b_\bE}(\acc(b_\bE))=V_{b_\bE}(\acc(b_\bE)).\end{equation} 
    By Lemma~\ref{lem:compatCond}, the classes $\{\bE_k\}_{k \geq 0}$ are adjacent to $\bE$ for all $k$ and further the recursion parameter is given by $p_1q_0-p_0q_1=\sqrt{p^2+q^2-6pq+8}$. In the language of \cite{ICERM}, this implies that $\bE$ is the  associated blocking class to $\{\bE_k\}_{k \geq}.$ Therefore, by \cite[Thm.52]{ICERM}, the left endpoint $b^\ell_\bE$ of the blocked interval $J_\bE$ is given by 
        \[b^\ell_\bE= \lim_{k \to \infty} \frac{m_k}{d_k},\] 
      and hence, $b^\ell_\bE=b_\bE$ and by \eqref{eq:embedMain}, the result follows. 
    \end{proof}

     The study of the tuples $\bE \in \Cc_n$ as obstructions to symplectic embeddings is done in \cite{MMW}. In fact, the authors of \cite{MMW} use Corollary~\ref{cor:4Dusa} to show that the classes $\bE \in \Cc_n$ are not only quasi-perfect classes but also represent embedded symplectic spheres, i.e. $\bE \in \Ee$, called {\bf perfect classes}. The work of \cite{MMW} shows that the tuples in $\Cc_n$ are the only perfect classes with $p/q \geq 6$.
\begin{remark} \label{rmk:cantor} \rm
    By \cite[Thm.1.1.1]{MMW}, for a fixed $n,$ the union of $J_{\bE}$ for $ \bE \in \Cc_n$ is homeomorphic to the complement of a middle third Cantor set. By Corollary~\ref{cor:4Dusa}, the $b$-values we consider in Theorem~\ref{thm:Main1} are the left endpoints of these intervals, see Figure~\ref{fig:tree}. If we consider the corresponding possible accumulation points blocked by $ \bE \in \Cc_n$, then these classes block a dense interval of $[2n+6,2n+8].$ 
 \end{remark}

If we suspect the embeddings from the base diagrams are optimal, it is not surprising that the numerics of the base diagrams relate to the numerics of the perfect classes, as the perfect classes give sharp obstructions. But it is still not known for what $b$-values and $z$-values the ATF base diagrams are giving sharp embeddings, i.e. computing $c_b(z)$. Hence, understanding the correspondence between the numerics might be useful. Here, we make a few observations about the perfect classes and the ATF base diagrams.\footnote{The anonymous referee comments inspired including some of these observations.} 

    \begin{itemize}
        \item For all $\bE \in \Cc_n,$ there is a triple such that $\Tt:=((d,m,p,q),\bE_\mu,\bE_\rho)$ for some $\bE_\mu,\bE_\rho \in \Cc_n$.  If we consider the base diagram $Q(\Tt)$, the corner at $Y$ can be seen in Figure~\ref{fig:RHB}. Following \cite[Sec.9.3]{S}, locally, the corner is an ATF base diagram for the rational homology ball $B_{q,p}$ with boundary the lens space $L(q^2,pq-1)$ in $H_b$ for $b$ such that $Q(\Tt)$ is defined. In particular, this holds for $b=b_{\bE_\rho}.$ 
        \item Removing the nodal rays of the diagrams in Definition~\ref{def:QT}, the diagrams corresponds to the moment polygon for a singular toric symplectic manifold. Here, the singularities are cyclic quotient singularities. For $Q(\Tt)$ where $\Tt=(\bE_\la,\bE_\mu,\bE_\rho) \in \Cc_n$, the singularity at the $X,V,Y$ corner of $Q(\Tt)$ is determined by $p_\la/q_\la,p_\mu/q_\mu,p_\rho/q_\rho$, respectively. While this statement is clear for the $X,Y$ corners, the statement holds for $V$ as well because by Proposition~\ref{prop:triple} and Definition~\ref{def:QT}, we know the image of the corner $V$ under $M_Y.$ This preserves the singularity type.  See \cite[Sec.7.4]{evanBook} for more details. 
        \item For a triple $\Tt:=(\bE_0,\bE_1,\bE),$ the base diagrams limiting to the triangle are $y^k(Q(\Tt))=Q(y^k\Tt).$ Letting $y^k\Tt=(\bE_k,\bE_{k+1},\bE),$ the work of \cite{MMW} using Corollary~\ref{cor:4Dusa} shows that the obstructive functions $\mu_{\bE_k,b_\bE}(z)$ are live near $z=p_k/q_k.$ The sequence of ATFs imply that these infinite staircases correspond to a sequence of embeddings: $ B_{q_k,p_k} \hookrightarrow H_{b_\bE}$
        where $B_{q_k,p_k}$ is the rational homology ball with boundary the lens space $L(q_k^2,p_kq_k-1).$ 
        \item As mentioned in Remark~\ref{rmk:future}~(iii), different $v$-mutations are expected to correspond to the symmetries. Hence, we expect similar relationships between the ATFs and the perfect classes $(d,m,p,q)$ with $p/q \leq 6$ studied in \cite{MMW}. In \cite{MMW}, all perfect classes with $3+2\sqrt{2} \leq p/q$ were found. 
    \end{itemize}

     \begin{center}
 \begin{figure}
 \begin{overpic}[unit=0.5mm]{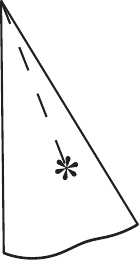}
 	\put (8,20) {$\begin{pmatrix} q \\ -p \end{pmatrix}$}
	\put (30,60) {$\begin{pmatrix} -q^2 \\ q p-1 \end{pmatrix}$}
  \end{overpic}
     \caption{As seen in \cite[Fig.13]{S}, this is an ATF for a rational ball with boundary $L(q^2,pq-1)$.}
     \label{fig:RHB}
 \end{figure}
 \end{center}


\begin{thebibliography}{CG-HMP}


\bibitem[BHM]{ICERM} M. Bertozzi, T. Holm, E. Maw, D. McDuff, G. Mwakyoma, A. R. Pires,and M. Weiler:  Infinite Staircases for Hirzebruch Surfaces, arXiv:2010.08567, Springer-Verlag, 2021.

\bibitem[BS]{BS} J. Brendel and F. Schlenk, Pinwheels as Lagrangian barriers, arXiv: 2210.00280.

\bibitem[BHO]{BHO} O. Buse, R. Hind, and E. Opshtein, Packing stability for symplectic four-manifolds, {\it Trans. Amer. Math. Soc.} {\bf 368} (2016), 8209-8222.

\bibitem[CV]{CV}  R. Casals and R. Vianna, Sharp ellipsoid embeddings and toric mutations, arXiv:2004.13232

\bibitem[CG]{CG} D. Cristofaro-Gardiner, Special eccentricities of rational four-dimensional ellipsoids, 
arXiv:2004.13647.

\bibitem[CG2]{CG2}  
D. Cristofaro-Gardiner, Symplectic embeddings from concave toric domains into convex ones, {\it J. Differential Geom.} {\bf 112}, (2019), 199-232. 

\bibitem[CG-HMP]{AADT}
D. Cristofaro-Gardiner, T. Holm, A. Mandini, and A. R. Pires,  On infinite staircases in toric symplectic four-manifolds, arxiv: 2004.07829.

\bibitem[E]{evanBook} J.D.Evans, Lectures on Lagrangian torus fibrations, arxiv:2110.08643v4.

\bibitem[ES]{ES} J. D. Evans and I. Smith. Markov numbers and Lagrangian cell complexes in the complex projective plane, {\it Geom. Topol. } {\bf 22}(2), 1143-1180, (2018).

\bibitem[EU]{EU} J. D. Evans and G. Urz\'ua, Antiflips, mutations, and unbounded symplectic embeddings of rational homology balls, {\it Ann. de I'Institut Fourier} {\bf 71} (2021). no.5, 1807-1843.

\bibitem[FHM]{REU} C. Farley, T. Holm, N. Magill, J. Schroder, M. Weiler, Z. Wang, and E. Zabelina, Four-periodic infinite staircases for four-dimensional polydisks, arXiv:2210.15069.

\bibitem[FM]{FM} D. Frenkel and D. M\"uller, Symplectic Embeddings of four-dimensional ellipsoids into cubes, {\it J. of Symplectic Topol.} {\bf 13}, (2015), 765--847. 

\bibitem[Hu]{Hutch} M. Hutchings, ``Quantitative embedded contact homology'', \textit{J. Diff. Geom.} 88(2):231–266, 2011.

\bibitem[LS]{LS} N. C. Leung and M. Symington. Almost toric symplectic four-manifolds. {\it J. Symplectic Geom.}, 8(2): 143-187,2010. 

\bibitem[LiLi]{li-li} B-H Li and T-J Li, ``Symplectic genus, minimal genus and diffeomorphisms,'' \textit{Asian J. Math.} 6:123-144, 2002.

\bibitem[LiLiu]{li-liu} T-J Li, A-K Liu, ``Uniqueness of symplectic canonical class, surface cone and symplectic cone of
4-manifolds with $b_+ = 1$,'' \textit{J. Differential Geom.} 58:331-370, 2001.

\bibitem[M]{M} N. Magill, Almost toric fibrations in 2-fold blowups of $\oCP^2$, in preparation.

\bibitem[MM]{MM} N. Magill and  D. McDuff, Staircase symmetries  in Hirzebruch surfaces, arXiv:2010.08567.


\bibitem[MMW]{MMW} N. Magill, D. McDuff, and M. Weiler, Staircase patterns in Hirzebruch surfaces, arxiv:2203.06453. 

\bibitem[Mc]{Mell} D. McDuff, Symplectic embedding os 4-dimensional ellipsoids, {\it J. Topol.} {\bf 8} (2015), no. 4, 1119-1122. 

\bibitem[MP]{MP} D. McDuff and L. Polterovich, Symplectic packings and algebraic geometry, {\it Invent. Math.} 115: 405-29. (1994). 

\bibitem[McS]{ball}
D. McDuff and F. Schlenk, The embedding capacity of 4-dimensional symplectic ellipsoids, {\it Ann. Math}  (2) 175 (2012), no. 3, 1191--1282.

\bibitem[S]{S}
M. Symington, Four dimensions from two in symplectic topology, In {\it Topology and geometry of Manifolds} (Athens, GA, 2001), volume 71 of {\it Proc. Sympos. Pure Math.} 153-208, Amer. Math Soc., Providence, RI, (2003).


\bibitem[U]{Usher}
M. Usher, Infinite staircases in the symplectic embedding problem for four-dimensional ellipsoids into polydisks, {\it Algebr. Geom. Topol.} 19(4):1935-2022, 2019.

\bibitem[V1]{V1} R. Vianna. Infinitely many exotic monotone Lagrangian tori in $\C P^2$. {\it J. Topol.}, 9(2): 535-551, (2016).

\bibitem[V2]{V2} R. Vianna. Infinitely many monotone Lagraingian tori in Del Pezzo surfaces. {\it Selecta Mathematica}, {\bf 23}: 1955-1996, (2017).
\end{thebibliography}
\end{document}